\documentclass[review,12pt]{elsarticle}
\usepackage{amsmath}
\usepackage{amsfonts}
\usepackage{geometry} 
\usepackage{xcolor}  
\usepackage[parfill]{parskip}    
\usepackage{lineno}
\usepackage{hyperref}

\hypersetup{colorlinks=true, urlcolor=red} 
\usepackage{graphicx}
\usepackage{amssymb}
\usepackage{epstopdf}
\usepackage{psfrag}
\usepackage{subfigure}

\newcommand {\bea}{\begin{eqnarray}}
\newcommand {\ea}{\end{eqnarray}}
\newenvironment{proof}[1][Proof]{\textbf{#1.} }{\hspace{\stretch{1}}\rule{0.5em}{0.5em}}

\usepackage{xspace}

\newtheorem{thm}{Theorem}
\newtheorem{lem}{Lemma}

\newtheorem{prp}{Proposition}

\newtheorem{ass}{Assumption}

\begin{document}
\begin{frontmatter}

\title{Strong error analysis  of a stochastic exponential integrator for SPDEs driven by fractional Brownian motion with $H<1/2$}
\author[at]{Antoine Tambue}
\ead{antonio@aims.ac.za}
\address[at]{Department of Computer science, Electrical engineering and Mathematical sciences,  Western Norway University of Applied Sciences, Inndalsveien 28, 5063 Bergen.}
\author[ajn]{Aurelien Junior Noupelah}
\address[ajn]{D\'epartement de Math\'ematiques et de Statistique, Universit\'e de Moncton, NB E1A 3E9, Canada}
\cortext[cor1]{Corresponding author}
\ead{aurelien.junior.noupelah@umoncton.ca}


\author[laf]{Louis Aim\'e Fono}
\ead{lfono2000@yahoo.fr}
\address[laf]{Equipe MASS, Laboratoire de Math\'ematiques,  UFD MIA, EDOSFA, Universit\'e de Douala, P.O. BOX 24157, Douala, Cameroun}

\begin{abstract}
Fractional Brownian motion provides a useful framework for modeling
random phenomena with anti-persistent temporal correlations that arise in
applications such as anomalous diffusion, hydrology, finance, and
geophysical processes. In this paper, we study a class of semilinear
stochastic partial differential equations driven by additive
fractional Brownian motion with Hurst parameter
$H\in(0,\frac12)$. We establish the well-posedness and space-time
regularity of the mild solution and investigate the strong convergence
of a stochastic exponential integrator for the temporal discretization.
The analysis allows the linear operator to be non-self-adjoint and
combines analytic semigroup estimates with the canonical Hilbert-space
structure associated with fractional Brownian motion and Malliavin
calculus to handle the low temporal regularity of the noise. We derive
an $H$-dependent strong temporal convergence rate, which reaches
$H+\frac12$ under maximal spatial regularity. Numerical experiments
for a stochastic advection--diffusion--reaction problem with
heterogeneous Darcy flow confirm the theoretical temporal convergence
rates. We also estimate the mean of the solution with different Hurst parameters $H\in(0,\frac12)$ using the Monte Carlo method.
\end{abstract}
\end{frontmatter}

\section{Introduction}
Uncertainty quantification  of transport phenomena in porous media with anti-persistent temporal correlations requires  sophistical mathematical techniques to accurately  approximate the  samples of the solution  of the corresponding Stochastic Partial Differential Equation (SPDE) model,  driven by 
fractional Brownian motion noise with Hurst parameter $H\in(0,\frac12)$.
In this regard, given a fixed final time $T>0$, we consider the numerical approximation of  the following  semilinear SPDE
\begin{equation}
\label{pb}
\left\{
\begin{array}{ll}
dX(t)+AX(t)\,dt
=
F(X(t))\,dt+\phi\,dB_Q^H(t),
& t\in(0,T],\\[1mm]
X(0)=X_0,
\end{array}
\right.
\end{equation}
in the Hilbert space $\mathcal{H}=L^2(\Lambda)$, where
$\Lambda\subset\mathbb{R}^d$, $d=1,2,3$, is a bounded domain.
Here, $-A$ generates an analytic semigroup
$S(t)=e^{-tA}$ on $\mathcal{H}$, with $A$ not necessarily
self-adjoint, $F:\mathcal{H}\rightarrow\mathcal{H}$ denotes the
nonlinear drift, and $\phi$ is a deterministic operator describing
the spatial regularity of the noise. The process $B_Q^H$ is an
$\mathcal{H}$-valued fractional Brownian motion (fBm) with covariance
operator $Q$ and Hurst parameter $H\in\left(0,\frac12\right)$.
Precise assumptions on the coefficients, the initial datum, and the
fractional noise are given in Section~\ref{mathsetting}. Under these
assumptions, the solution of \eqref{pb} is understood in the mild
sense as
\begin{equation}
\label{mild}
X(t)
=
S(t)X_0
+
\int_0^t S(t-s)F(X(s))\,ds
+
\int_0^t S(t-s)\phi\,dB_Q^H(s),
\qquad t\in[0,T].
\end{equation}

In recent years, considerable progress has been made in stochastic
calculus and stochastic differential equations driven by fractional
Brownian motion, particularly for $H>\frac12$; see, for instance,
\cite{Alo,Car,Mem,Mis}. The theory of stochastic partial differential
equations driven by infinite-dimensional fractional Brownian motion
has also been investigated in several works; see
\cite{Car,Duna,Dunb,Lia,Lib,Lic} and the references therein.
These contributions provide important results on the existence,
uniqueness, and regularity of solutions to stochastic evolution
equations with fractional Gaussian forcing.

The numerical approximation of SPDEs driven by fractional Brownian
motion has also attracted increasing attention. Most of the available
strong convergence results, however, concern the smoother regime
$H>\frac12$. Wang, Qi and Jiang \cite{Wanc} established sharp
mean-square spatial and temporal regularity results for semilinear
parabolic SPDEs driven by infinite-dimensional fBm and obtained
optimal convergence rates for a spectral Galerkin approximation
combined with a linear implicit Euler method. In the finite element
setting, Noupelah and Tambue \cite{Nou} studied semilinear SPDEs
driven by additive fBm and a Poisson random measure. They considered
linear implicit Euler, exponential, and stochastic exponential
Rosenbrock time-stepping schemes and allowed the linear operator to
be non-self-adjoint. More recently, Chen, Xia and Yan
\cite{ChenXiaYan2026} considered linear elliptic and parabolic SPDEs
driven by fractional noise and developed numerical approximations
based on difference and finite element methods.

Exponential integrators have more recently received particular
attention for equations driven by fractional Brownian motion.
Kamrani, Debrabant and Jamshidi
\cite{KamraniDebrabantJamshidi2024} analyzed an exponential Euler
method for stiff stochastic differential equations with additive
fractional Brownian noise. For infinite-dimensional problems,
Kamrani \cite{Kamrani2025} combined a spectral Galerkin
approximation with an exponential Euler method for SPDEs driven by
fBm with $H>\frac12$ and derived strong convergence estimates for
the resulting fully discrete approximation. Very recently, Kamrani
and Debrabant \cite{KamraniDebrabant2026} investigated exponential
Euler schemes for systems of fractional stochastic differential
equations arising from spatial discretizations of semilinear SPDEs.
Using Malliavin calculus, they obtained an improved strong convergence
order for one of the considered schemes. This result demonstrates
that a refined stochastic analysis can lead to temporal convergence
rates beyond those suggested by the H\"older regularity of the
fractional Brownian paths alone.

Despite these developments, the numerical analysis of semilinear
parabolic SPDEs in the rough regime $H\in\left(0,\frac12\right)$
remains substantially more delicate. In contrast to the case
$H>\frac12$, the increments of fBm are negatively correlated and its
sample paths have lower temporal regularity. Moreover, several
stochastic estimates commonly used in the smoother fractional regime
cannot be directly extended to $H<\frac12$; in particular, the
estimate in \cite[Lemma 1]{Car}, which has been extensively used in
the analysis for $H>\frac12$, is not applicable in the present
setting. The treatment of the stochastic terms therefore requires
estimates adapted to the canonical Hilbert space associated with fBm
and, for random integrands, tools from Malliavin calculus and
divergence-type stochastic integration; see
\cite{Bar,Mis,Nualart,DecreusefondUstunel}.

Although the rough fractional regime is more challenging, some
numerical results are available. Cao, Hong and Liu \cite{Cao}
investigated a Galerkin finite element approximation for a
one-dimensional second-order stochastic differential equation of
boundary value type driven by fBm with $H\leq\frac12$. More recently,
Ding and Wang \cite{Din} studied stochastic evolution equations
driven by multiplicative fBm, including the case
$H\in(0,\frac12)$, and established strong convergence results for a
fully discrete scheme based on a spectral Galerkin approximation in
space and a time-stepping method in time.

A further limitation of several existing numerical analyses is the
assumption that the linear operator is self-adjoint, which naturally
leads to spectral Galerkin approximations. Such a framework can be
restrictive for second-order differential operators involving, for
instance, advection terms. Finite element methods provide greater
flexibility in this respect, as they allow the treatment of
non-self-adjoint operators without requiring an explicit spectral
decomposition.

Motivated by these observations, we investigate the strong numerical
approximation of \eqref{pb} in the rough fractional regime
$H<\frac12$. We combine a standard finite element discretization in
space with a stochastic exponential integrator in time. The main
challenge is to derive a sharp temporal error estimate despite the low
regularity of the fractional noise. To overcome this difficulty, we
combine analytic semigroup estimates with the canonical Hilbert-space
structure associated with fBm and Malliavin-calculus arguments. This
refined analysis allows us to control the stochastic terms arising
from the nonlinear drift and to improve the temporal strong
convergence rate.

The main contributions of this work can be summarized as follows:
\begin{itemize}

\item We develop a set of auxiliary estimates specifically tailored to
the temporal discretization in the rough fractional regime
$H\in(0,\frac12)$; see Section \ref{auxi}. A key ingredient is a diagonal--off-diagonal decomposition of the fractional covariance terms, combined with the summability of the discrete fractional covariance sequence, which
yields the required quadrature estimate for the stochastic convolution
in Lemma \ref{stoch_quad_error}.

\item To handle the nonlinear stochastic contribution, we establish a
uniform \\Malliavin-derivative estimate for the semidiscrete solution
(Lemma \ref{lem_malliavin_Xh}) and a cellwise fractional covariance estimate adapted to the temporal grid (Lemma \ref{lem_cellwise_covariance}). These estimates allow the random integrands arising from the Taylor expansion of the nonlinear drift to be treated as divergence integrals. Their combination in Lemma \ref{improved_drift_defect} provides the
crucial control of the accumulated nonlinear drift defect without
losing the temporal order.

\item Building on these auxiliary estimates, we establish the strong
convergence of the stochastic exponential integrator. The resulting
temporal convergence rate depends explicitly on the Hurst parameter
and reaches $H+\frac12$ under maximal spatial regularity.

\end{itemize}

The remainder of the paper is organized as follows.
Section~\ref{mathsetting} introduces the functional and probabilistic
framework and establishes the well-posedness of the problem.
Section~\ref{reg} derives spatial and temporal regularity
estimates for the mild solution. Section~\ref{spa_conv} presents the
finite element spatial discretization and establishes the corresponding
strong convergence rate. Section~\ref{schemes} introduces the fully
discrete stochastic exponential scheme, develops the auxiliary
estimates required for the temporal error analysis, and establishes
the main strong convergence result. Finally,
Section~\ref{numerik} presents numerical simulations illustrating
the theoretical convergence rate  and the estimation of the mean of the solution with different  Hurst parameters $H\in(0,\frac12)$ using the Monte Carlo method.

\section{Functional framework, fractional noise and well-posedness}
\label{mathsetting}

In this section, we introduce the functional and probabilistic
framework required for the analysis and establish the well-posedness
of problem \eqref{pb}.

\subsection{Functional analytic setting}

Let $\Lambda\subset\mathbb R^d$ be a bounded domain with sufficiently
smooth boundary and set $\mathcal H=L^2(\Lambda)$.
Depending on the boundary conditions, we define
$V=
\begin{cases}
H_0^1(\Lambda),
& \text{for homogeneous Dirichlet boundary conditions},\\[1mm]
H^1(\Lambda),
& \text{for Neumann or Robin boundary conditions}.
\end{cases}
$
Identifying $\mathcal H$ with its dual, we obtain the Gelfand triple $V\hookrightarrow\mathcal H\equiv\mathcal H^*
\hookrightarrow V^*$,
where the embeddings are continuous and dense. This variational
framework is standard for second-order elliptic and parabolic
problems; see, for instance, \cite{Fuj,LionsMagenes,Thomee}.

Consider the second-order differential operator
\[
\mathcal Au
=
-\sum_{i,j=1}^{d}
\frac{\partial}{\partial x_i}
\left(
D_{ij}(x)\frac{\partial u}{\partial x_j}
\right)
+
\sum_{i=1}^{d}
q_i(x)\frac{\partial u}{\partial x_i},
\]
where $D_{ij},q_i\in L^\infty(\Lambda)$. We assume that the diffusion
matrix $\mathbf D=(D_{ij})_{1\leq i,j\leq d}$
is uniformly elliptic, namely, there exists $c_1>0$ such that
\begin{equation}
\label{ellipticity}
\sum_{i,j=1}^{d}
D_{ij}(x)\xi_i\xi_j
\geq
c_1|\xi|^2,
\qquad
\xi\in\mathbb R^d,\quad
x\in\overline{\Lambda}.
\end{equation}

The operator $\mathcal A$ is associated with the bilinear form
\begin{equation*}
a(u,v)
=
\int_\Lambda
\left(
\sum_{i,j=1}^{d}
D_{ij}(x)
\frac{\partial u}{\partial x_i}
\frac{\partial v}{\partial x_j}
+
\sum_{i=1}^{d}
q_i(x)
\frac{\partial u}{\partial x_i}v
\right)dx,
\qquad
u,v\in V,
\end{equation*}
with the usual additional boundary contribution in the Robin case.
Under the ellipticity condition \eqref{ellipticity}, the form
$a(\cdot,\cdot)$ satisfies a G{\aa}rding inequality. After a standard
shift of the operator, if necessary, we may assume that the form is
coercive: 
\[
a(v,v)\geq c\|v\|_V^2, \qquad v\in V.
\]
The variational realization of such elliptic operators and its
connection with parabolic problems are discussed in
\cite{Fuj,LionsMagenes,Thomee}. We denote by $A$ the realization of $\mathcal A$ in $\mathcal H$. Under the preceding assumptions, $A$ is sectorial and $-A$ generates a bounded analytic semigroup $S(t)=e^{-tA}$, $t\geq0$.
The theory of sectorial operators and analytic semigroups can be found
in \cite[Chapter 2]{Pazy}, \cite[Sections 1.3--1.4]{Henry}, and
\cite[Appendix A]{DaPratoZabczyk}. In particular, analytic semigroups
provide the smoothing properties required below and allow fractional
powers of $A$ to be introduced in a natural way.
For $\alpha\geq0$, we define $\mathcal H^\alpha:=\mathcal D(A^{\alpha/2})$, $\|v\|_\alpha:=\|A^{\alpha/2}v\|$.
Under standard elliptic regularity assumptions, these fractional
domain spaces are equivalent, up to equivalent norms, to appropriate
Sobolev spaces satisfying the corresponding boundary conditions;
see \cite{Fuj,Henry,Thomee}. The following analytic semigroup estimates will be used repeatedly.

\begin{prp}
\label{semigroup}
Let $S(t)=e^{-tA}$ be the bounded analytic semigroup generated by
$-A$. For every $\delta\geq0$ and $0\leq\gamma\leq1$, there exists
$C>0$ such that
\[
\|A^\delta S(t)\|_{L(\mathcal H)}
\leq
Ct^{-\delta}
\qquad
\text{and} \qquad
\|A^{-\gamma}(I-S(t))\|_{L(\mathcal H)}
\leq
Ct^\gamma,
\qquad
t>0.
\]
Moreover, for every $\delta\geq0$, $A^\delta S(t)v=S(t)A^\delta v$,
$v\in\mathcal D(A^\delta)$, $t\geq0$.
\end{prp}

We next introduce the operator spaces used to describe the covariance
structure of the infinite-dimensional noise. Let $U$ and $K$ be
separable Hilbert spaces. We denote by $L(U,K)$ the space of bounded
linear operators from $U$ into $K$ and by $\mathcal L_2(U,K)$ the
space of Hilbert--Schmidt operators. For
$\Phi\in\mathcal L_2(U,K)$,
\[
\|\Phi\|_{\mathcal L_2(U,K)}^2
=
\sum_{j=1}^{\infty}\|\Phi e_j\|_K^2,
\]
where $(e_j)_{j\geq1}$ is any orthonormal basis of $U$. This
definition is independent of the chosen basis.

Let $Q:\mathcal H\rightarrow\mathcal H$ be a nonnegative
self-adjoint operator. We define $L_2^0:=\mathcal L_2\bigl(Q^{1/2}(\mathcal H),\mathcal H\bigr)$ equipped with
\[
\|\Phi\|_{L_2^0}^2
=
\|\Phi Q^{1/2}\|_{\mathcal L_2(\mathcal H)}^2
=
\sum_{j=1}^{\infty}
\|\Phi Q^{1/2}e_j\|^2.
\]
We refer to \cite{DaPratoZabczyk} for the theory of covariance
operators, Hilbert--Schmidt operators and infinite-dimensional
stochastic integration. Having fixed the deterministic operator framework, we now introduce
the fractional Gaussian noise driving the equation.

\subsection{Fractional Brownian motion and the canonical Hilbert space}

Let $H\in(0,\frac12)$. A real-valued fractional Brownian motion
$\{\beta^H(t)\}_{t\in[0,T]}$ is a centered Gaussian process with
covariance
\[
R_H(t,s)
:=
\mathbb E\bigl[\beta^H(t)\beta^H(s)\bigr]
=
\frac12
\left(
t^{2H}+s^{2H}-|t-s|^{2H}
\right),
\qquad
s,t\in[0,T].
\]
General properties of fractional Brownian motion, including
self-similarity, stationary increments and its Gaussian structure,
can be found in \cite{Mis,Nualart,DecreusefondUstunel}. To introduce the deterministic Wiener integral, let $\mathcal E$
denote the space of real-valued step functions on $[0,T]$. For $\displaystyle f
=\sum_{i=0}^{n-1}f_i\mathbf 1_{[t_i,t_{i+1})}$,
$0=t_0<t_1<\cdots<t_n=T$,
set
\[
I_1(f)
:=
\sum_{i=0}^{n-1}
f_i
\bigl(
\beta^H(t_{i+1})-\beta^H(t_i)
\bigr).
\]
This is the natural first-order integral associated with the
increments of $\beta^H$; see \cite{Bar,Mis}. The canonical Hilbert
space $\mathfrak H_H$ associated with $\beta^H$ is defined as the
completion of $\mathcal E$ with respect to $\displaystyle \langle f,g\rangle_{\mathfrak H_H}
:=
\mathbb E\bigl[I_1(f)I_1(g)\bigr]$.
Consequently,
$\|f\|_{\mathfrak H_H}^2=\mathbb E|I_1(f)|^2$.
The mapping $I_1$ therefore extends by isometry from
$\mathfrak H_H$ into the Gaussian space generated by the fractional
Brownian motion; see \cite{Nualart,DecreusefondUstunel,Mis}.

For $H<\frac12$, the canonical integrand space is more restrictive
than in the standard Brownian case. In particular, its structure
reflects the singular covariance kernel generated by the roughness
of the fractional Brownian paths. A useful explicit characterization
was obtained by Bardina and Jolis \cite{Bar}. For suitable
deterministic functions $f,g:[0,T]\rightarrow\mathbb R$,
\begin{align}
\langle f,g\rangle_{\mathfrak H_H}
&=
\frac{H(1-2H)}{2}
\int_0^T\int_0^T
(f(u)-f(v))(g(u)-g(v))
|u-v|^{2H-2}\,du\,dv
\nonumber\\
&\quad+
H\int_0^T
f(s)g(s)
\left(
s^{2H-1}+(T-s)^{2H-1}
\right)ds.
\label{canonical_inner_product}
\end{align}

The double integral in \eqref{canonical_inner_product} measures
fractional increments of the integrand, whereas the second term
accounts for the interaction with the endpoints of $[0,T]$. This
representation is particularly well suited to the cellwise estimates
used later in the temporal convergence analysis. More generally, if $V$ is a separable Hilbert space, we denote by
$\mathfrak H_H(V)$ the $V$-valued canonical Hilbert space obtained by
completion of the $V$-valued step functions with respect to the
corresponding covariance norm; see \cite{Nualart,Mis}. The scalar
identity \eqref{canonical_inner_product} extends componentwise, and
for suitable $g:[0,T]\rightarrow V$ one has
\begin{align}
\|g\|_{\mathfrak H_H(V)}^2
&=
\frac{H(1-2H)}{2}
\int_0^T\int_0^T
\|g(u)-g(v)\|_V^2
|u-v|^{2H-2}\,du\,dv
\nonumber\\
&\quad+
H\int_0^T
\|g(s)\|_V^2
\left(
s^{2H-1}+(T-s)^{2H-1}
\right)ds.
\label{canonical_V_norm}
\end{align}

We now introduce the Hilbert-space-valued fractional Brownian motion.
Let $(q_j,e_j)_{j\geq1}$ denote the eigenpairs of $Q$. The
fractional $Q$-Brownian motion is represented formally as
\[
B_Q^H(t)
=
\sum_{j=1}^{\infty}
q_j^{1/2}\beta_j^H(t)e_j,
\qquad
t\in[0,T],
\]
where $(\beta_j^H)_{j\geq1}$ is a sequence of independent
real-valued fractional Brownian motions with Hurst parameter $H$.
The corresponding infinite-dimensional construction is the
fractional analogue of the standard $Q$-Wiener expansion used in
Hilbert-space stochastic analysis; see
\cite{DaPratoZabczyk,Duna,Dunb}.

For a deterministic function
$\Phi:[0,T]\rightarrow L_2^0$ belonging to the appropriate canonical
integrand space, we define
\begin{equation}
\label{QfBm_integral}
\int_0^T\Phi(s)\,dB_Q^H(s)
:=
\sum_{j=1}^{\infty}
\int_0^T
\Phi(s)Q^{1/2}e_j\,d\beta_j^H(s),
\end{equation}
whenever the series converges in $L^2(\Omega;\mathcal H)$.
Different constructions and properties of stochastic integration
with respect to fractional Brownian motion in finite- and
infinite-dimensional spaces can be found in
\cite{Alo,Duna,Dunb,Mis,DecreusefondUstunel}.
Combining the scalar covariance representation
\eqref{canonical_inner_product} with the series expansion
\eqref{QfBm_integral} yields the following $L_2^0$-valued identity.

\begin{lem}
\label{fracint}
Let $H\in(0,\frac12)$ and let
$\Phi:[0,T]\rightarrow L_2^0$ be a deterministic function belonging
to the canonical integrand space associated with $B_Q^H$. Then
\[
\begin{aligned}
\mathbb E
\left\|
\int_0^T\Phi(s)\,dB_Q^H(s)
\right\|^2
&=
\frac{H(1-2H)}{2}
\int_0^T\int_0^T
\|\Phi(u)-\Phi(v)\|_{L_2^0}^2
|u-v|^{2H-2}\,du\,dv
\nonumber\\
&\quad+
H\int_0^T
\|\Phi(s)\|_{L_2^0}^2
\left(
s^{2H-1}+(T-s)^{2H-1}
\right)ds.
\label{fractional_isometry}
\end{aligned}
\]
\end{lem}

Lemma~\ref{fracint} is sufficient for deterministic stochastic
convolutions. In the refined temporal analysis, however, the
nonlinear Taylor expansion generates random coefficients inside
fractional stochastic integrals. To handle such terms, we need the
Malliavin derivative and the associated divergence integral.

\subsection{Malliavin derivative and divergence integral}

We recall the elements of Malliavin calculus that will be used below.
The general Gaussian theory is presented in \cite{Nualart}, while its
formulation for fractional Brownian motion is developed in
\cite{DecreusefondUstunel,Mis}. Let $\mathcal S$ denote the class of smooth cylindrical random
variables of the form $F=f\bigl(B^H(h_1),\ldots,B^H(h_n)\bigr)$,
where $h_1,\ldots,h_n\in\mathfrak H_H$ and
$f\in C_b^\infty(\mathbb R^n)$. The Malliavin derivative of $F$ with
respect to $B^H$ is defined by
\[
D^H F
=
\sum_{j=1}^{n}
\partial_j f
\bigl(
B^H(h_1),\ldots,B^H(h_n)
\bigr)h_j.
\]
The operator $D^H$ is closable in $L^p(\Omega)$ for $p\geq1$;
see \cite{Nualart,Mis}. This allows one to define the associated
Sobolev--Malliavin spaces. For $p\geq1$, we denote by $\mathbb D_H^{1,p}$ the closure of
$\mathcal S$ with respect to
\[
\|F\|_{\mathbb D_H^{1,p}}
=
\left(
\mathbb E|F|^p
+
\mathbb E
\|D^H F\|_{\mathfrak H_H}^p
\right)^{1/p}.
\]
If $V$ is a separable Hilbert space, the corresponding space
$\mathbb D_H^{1,p}(V)$ is defined analogously and $D^H F
\in
L^p\bigl(\Omega;\mathfrak H_H(V)\bigr)$.
The spaces $\mathbb D_H^{1,p}(\mathcal H)$ and
$\mathbb D_H^{1,p}(L_2^0)$ will occur naturally in the estimates of
the Malliavin derivatives of the semidiscrete solution and of the
random operator coefficients appearing in the local defects.

The divergence operator $\delta^H$ is defined as the adjoint of
$D^H$. More precisely,
$u\in L^2(\Omega;\mathfrak H_H)$ belongs to
$\operatorname{Dom}(\delta^H)$ if there exists $C_u>0$ such that
\[
\left|
\mathbb E
\langle D^H F,u\rangle_{\mathfrak H_H}
\right|
\leq
C_u
\|F\|_{L^2(\Omega)}
\]
for all $F\in\mathbb D_H^{1,2}$. In this case,
$\delta^H(u)\in L^2(\Omega)$ is uniquely characterized by
\[
\mathbb E\bigl[F\delta^H(u)\bigr]
=
\mathbb E
\left[
\langle D^H F,u\rangle_{\mathfrak H_H}
\right],
\qquad
F\in\mathbb D_H^{1,2}.
\]
This duality definition is the fractional counterpart of the
standard divergence operator in Gaussian Malliavin calculus; see
\cite[Definition 1.3.1]{Nualart} and \cite{Mis}. For deterministic $u\in\mathfrak H_H$, the divergence coincides with
the corresponding Wiener integral:
\[
\delta^H(u)
=
B^H(u)
=
\int_0^T u(s)\,dB^H(s).
\]
For random $u$, the quantity $\delta^H(u)$ is called the divergence,
or Skorohod, integral and is written
\[
\delta^H(u)
=
\int_0^T u(s)\,\delta B^H(s).
\]
The relation between the fractional divergence operator and the
classical Skorohod integral can also be described through the transfer
operator associated with the fBm; see \cite{Nualart,Mis}. In
particular, Mishura shows that the divergence operator can be
transported to the divergence operator associated with the underlying
standard Wiener process.

A key property used in our analysis is the continuity of the
divergence operator on the first-order Sobolev--Malliavin space.
For $u\in\mathbb D_H^{1,2}(\mathfrak H_H)$,
\[
\mathbb E|\delta^H(u)|^2
\leq
\mathbb E\|u\|_{\mathfrak H_H}^2
+
\mathbb E
\|D^H u\|_{\mathfrak H_H\otimes\mathfrak H_H}^2.
\]
This estimate follows from the general second-moment identity for
the divergence operator; see
\cite[Proposition 1.3.1]{Nualart}, and from the corresponding
fractional formulation in \cite{Mis}.
The estimate extends componentwise to Hilbert-space-valued
integrands. Hence, if $V$ is a separable Hilbert space and
$u\in\mathbb D_H^{1,2}(\mathfrak H_H(V))$, then
\begin{equation}
\label{divergence_estimate}
\|\delta^H(u)\|_{L^2(\Omega;V)}^2
\leq
C
\left(
\mathbb E
\|u\|_{\mathfrak H_H(V)}^2
+
\mathbb E
\|D^H u\|_{\mathfrak H_H\otimes\mathfrak H_H(V)}^2
\right).
\end{equation}
Thus, unlike the deterministic Wiener integral, a divergence integral
with random integrand requires control not only of the canonical norm
of the integrand itself but also of its Malliavin derivative. This is
precisely the mechanism that will be used later for random kernels
arising in the accumulated nonlinear drift defect.

We shall also repeatedly use the classical Beta-function identity
\cite{AbramowitzStegun}:
\[
\mathcal B(x,y)
=
\int_0^1
s^{x-1}(1-s)^{y-1}\,ds
=
\frac{\Gamma(x)\Gamma(y)}{\Gamma(x+y)},
\qquad
x,y>0.
\]
Consequently, after the change of variables
$s=t_1+(t_2-t_1)r$,
\begin{equation}
\label{beta_identity}
\int_{t_1}^{t_2}
(t_2-s)^{x-1}
(s-t_1)^{y-1}\,ds
=
\mathcal B(x,y)
(t_2-t_1)^{x+y-1},
\qquad
0\leq t_1<t_2\leq T.
\end{equation}

The preceding analytic and probabilistic tools allow us to formulate
the assumptions under which the stochastic evolution problem is
studied.

\subsection{Assumptions and well-posedness}

To ensure the existence and the \\
uniqueness of solution of \eqref{pb} and for the purpose of the convergence analysis, we make the following assumptions.

\begin{ass}[Initial data]
\label{initdata}
The initial datum
$X_0:\Omega\rightarrow\mathcal H$ is
$\mathcal F_0/\mathcal B(\mathcal H)$-measurable and satisfies $X_0\in L^2\left(\Omega;\mathcal D\left(A^{\frac{2H+\beta}{2}}\right)\right)$, with $1-2H <\beta\leq 1$.
\end{ass}

\begin{ass}[Additive fractional noise]
\label{noise}
The noise coefficient $\phi$ is deterministic and satisfies
\[
\|A^\kappa\phi\|_{L_2^0}\leq C_{\kappa},
\qquad 0\leq\kappa\leq1,
\]
for some constant $C_{\kappa}>0$.
\end{ass}
\begin{ass}[Nonlinearity]
\label{non}
The mapping $F:\mathcal H\rightarrow\mathcal H$ belongs to
$C_b^2(\mathcal H;\mathcal H)$. More precisely, there exists
$L>0$ such that
\begin{equation}
\label{ass_non_lin}
\|F(0)\|\leq L,
\qquad
\sup_{v\in\mathcal H}
\|F'(v)\|_{L(\mathcal H)}
\leq L,
\qquad
\sup_{v\in\mathcal H}
\|F''(v)\|_{L^{(2)}(\mathcal H;\mathcal H)}
\leq L.
\end{equation}
By the mean-value theorem in Banach spaces,
\eqref{ass_non_lin} implies the global Lipschitz estimate
\begin{equation}
\label{ass_non_lin_lipschitz}
\|F(u)-F(v)\|
\leq
L\|u-v\|,
\qquad
u,v\in\mathcal H,
\end{equation}
and hence
\begin{equation}
\label{ass_non_lin3}
\|F(v)\|
\leq
L(1+\|v\|),
\qquad
v\in\mathcal H.
\end{equation}
The boundedness of $F''$ will be required later when differentiating
the random Taylor coefficient with respect to the underlying
fractional Gaussian noise.
\end{ass}

We are now ready to state the existence and uniqueness result for the
mild solution of \eqref{pb}, given by \eqref{mild}.

\begin{thm}
\label{well_poss_pbfrac}
Under Assumptions~\ref{initdata}--\ref{non}, the SPDE \eqref{pb}
admits a unique mild solution $X\in C\bigl([0,T];L^2(\Omega;\mathcal H)\bigr)$ given by \eqref{mild}. Moreover,
\begin{equation}
\label{mean_square_bound}
\sup_{0\leq t\leq T}
\mathbb E\|X(t)\|^2
<\infty.
\end{equation}
\end{thm}

\begin{proof}

We prove the existence and uniqueness of the mild solution by means of
the Banach fixed-point theorem. Let $\mathcal{X}_T:=C\left([0,T];L^2(\Omega;\mathcal{H})\right)$, equipped with the norm $\|x\|_{\mathcal{X}_T}
:= \sup\limits_{0\leq t\leq T} \|x(t)\|_{L^2(\Omega;\mathcal{H})}$. For $u>0$, we also introduce the weighted norm
$\|x\|_{\mathcal{X}_T,u}:=\sup\limits_{0\leq t\leq T}e^{-ut}\|x(t)\|_{L^2(\Omega;\mathcal{H})}$.
Since $T<\infty$, the norms
$\|\cdot\|_{\mathcal{X}_T}$ and
$\|\cdot\|_{\mathcal{X}_T,u}$ are equivalent. Indeed,
\[
e^{-uT}
\|x\|_{\mathcal{X}_T}
\leq
\|x\|_{\mathcal{X}_T,u}
\leq
\|x\|_{\mathcal{X}_T}.
\]
Define the mapping $\xi:\mathcal{X}_T\longrightarrow\mathcal{X}_T$ by
\begin{equation}
\label{op}
(\xi x)(t)
=
S(t)X_0
+
\int_0^t S(t-s)F(x(s))\,ds
+
\int_0^t S(t-s)\phi\,dB_Q^H(s).
\end{equation}
The proof is divided into two steps.

\textbf{Step 1: $\xi$ maps $\mathcal{X}_T$ into itself.}

Let $x\in\mathcal{X}_T$. Using \eqref{op}, the triangle inequality and $(a+b+c)^2 \leq 3(a^2+b^2+c^2)$,
we obtain
\begin{eqnarray*}
	\mathbb{E}[\left\|(\xi x)(t)\right\|^2]
	&\leq& 3\mathbb{E}\left\|S(t)x_0\right\|^2+\mathbb{E}\left\|\int_0^tS(t-s)F(x(s))ds\right\|^2+ 3\mathbb{E}\left\|\int_0^t S(t-s)\phi dB^H(s) \right\|^2\nonumber\\
	&=:& 3\sum_{i=1}^3 I_i(t).
	\end{eqnarray*}

We estimate the three terms separately. By the stability property of the semigroup $S(t)$, see
Proposition \ref{semigroup} with $\delta=0$, $\|S(t)\|_{L(\mathcal{H})}
\leq C$. Therefore,
\begin{align}
I_1(t) = \mathbb{E}\left[\|S(t)X_0\|^2\right]\nonumber \leq C^2\mathbb{E}\left[\|X_0\|^2\right]<\infty.
\label{I1_est}
\end{align}
By the Cauchy--Schwarz inequality, the boundedness of the semigroup and the linear growth condition
\eqref{ass_non_lin3}, we have
\[
\begin{aligned}
I_2(t) &= \mathbb{E}\left[\left\|\int_0^tS(t-s)F(x(s))\,ds\right\|^2\right]\nonumber\\
&\leq C^2t\int_0^t\mathbb{E}\left[\|F(x(s))\|^2\right]ds\nonumber\\
&\leq C t\int_0^t\left(1+\mathbb{E}\|x(s)\|^2\right)ds\nonumber\\
&\leq C t\left(t+t\|x\|_{\mathcal{X}_T}^2\right) \leq C_T\left(1+\|x\|_{\mathcal{X}_T}^2\right),
\label{I2_est}
\end{aligned}
\]
where $C_T>0$ may depend on $T$ but is independent of $x$.\\
We now estimate $I_3(t)$. Applying Lemma \ref{fracint}, we obtain
\[
\begin{aligned}
I_3(t)
&=
\mathbb{E}
\left[
\left\|
\int_0^t
S(t-s)\phi\,dB_Q^H(s)
\right\|^2
\right]
\nonumber
\\
&\leq
\frac{H(1-2H)}{2}
\int_0^t
\int_0^t
\|
S(t-u)\phi
-
S(t-v)\phi
\|_{L_2^0}^2
|u-v|^{2H-2}
\,dv\,du
\nonumber
\\
&\quad
+
H
\int_0^t
\|S(t-s)\phi\|_{L_2^0}^2
\left(
s^{2H-1}
+
(t-s)^{2H-1}
\right)ds
\nonumber
\\
&=:I_{31}(t)+I_{32}(t).
\label{I3split}
\end{aligned}
\]
For $0\leq v\leq u\leq t$, the semigroup property yields $S(t-u)\phi-S(t-v)\phi=S(t-u)\left(I-S(u-v)\right)\phi$.
Consequently, 
\[
\|S(t-u)\phi-S(t-v)\phi\|_{L_2^0}\leq\|S(t-u)\|_{L(\mathcal{H})}\|A^{-\frac{\beta}{2}}\left(I-S(u-v)\right)\|_{L(\mathcal{H})}\|A^{\frac{\beta}{2}}\phi\|_{L_2^0}.
\]
Using Proposition \ref{semigroup} with
$\gamma=\frac{\beta}{2}$,
together with Assumption \ref{noise}, gives $ \displaystyle
\|S(t-u)\phi-S(t-v)\phi
\|_{L_2^0}
\leq
C
(u-v)^{\frac{\beta}{2}}
\|
A^{\frac{\beta}{2}}\phi
\|_{L_2^0}$.
Similarly, for $0\leq u\leq v\leq t$,$\displaystyle
\|
S(t-u)\phi-S(t-v)\phi
\|_{L_2^0}
\leq
C
(v-u)^{\frac{\beta}{2}}
\|
A^{\frac{\beta}{2}}\phi
\|_{L_2^0}$.
Therefore,
\begin{align*}
I_{31}(t)
&\leq
C
\|
A^{\frac{\beta}{2}}\phi
\|_{L_2^0}^2
\int_0^t
\int_0^t
|u-v|^{2H+\beta-2}
\,dv\,du.
\end{align*}
Since $1-2H<\beta$, we have $2H+\beta-1>0$. Hence the double integral is finite. More precisely,
\begin{align*}
\int_0^t
\int_0^t
|u-v|^{2H+\beta-2}
\,dv\,du
&=
2
\int_0^t
\int_0^u
(u-v)^{2H+\beta-2}
\,dv\,du
\nonumber
\\
&=
\frac{2}
{2H+\beta-1}
\int_0^t
u^{2H+\beta-1}
\,du
\nonumber
\\
&=
\frac{2t^{2H+\beta}}
{(2H+\beta-1)(2H+\beta)}.
\end{align*}
Therefore,
\begin{equation}
I_{31}(t)
\leq
C
t^{2H+\beta}
\|
A^{\frac{\beta}{2}}\phi
\|_{L_2^0}^2
<\infty.
\label{I31_est}
\end{equation}
We now estimate $I_{32}(t)$. Using again the boundedness of the
semigroup and Assumption \ref{noise} with $\kappa=0$, we obtain
\begin{align}
I_{32}(t)
&=
H
\int_0^t
\|S(t-s)\phi\|_{L_2^0}^2
\left(
s^{2H-1}
+
(t-s)^{2H-1}
\right)ds
\nonumber
\\
&\leq
C H
\|\phi\|_{L_2^0}^2
\int_0^t
\left(
s^{2H-1}
+
(t-s)^{2H-1}
\right)ds
\nonumber
\\
&=
C
\|\phi\|_{L_2^0}^2
t^{2H}.
\label{I32_est}
\end{align}
Combining \eqref{I31_est} and \eqref{I32_est}, it follows that
\[
I_3(t)
\leq
C
\left(
t^{2H+\beta}
\|
A^{\frac{\beta}{2}}\phi
\|_{L_2^0}^2
+
t^{2H}
\|\phi\|_{L_2^0}^2
\right).
\]

Therefore, for every $t\in[0,T]$,
\begin{align}
\mathbb{E}
\left[
\|(\xi x)(t)\|^2
\right]
&\leq
C_T
\left(
1
+
\mathbb{E}\|X_0\|^2
+
\|x\|_{\mathcal{X}_T}^2
+
\|\phi\|_{L_2^0}^2
+
\left\|
A^{\frac{\beta}{2}}\phi
\right\|_{L_2^0}^2
\right).
\label{xi_final_step1}
\end{align}
Hence, by Assumptions~\ref{initdata} and~\ref{noise}, and since
$x\in\mathcal{X}_T$, all the terms on the right-hand side of
\eqref{xi_final_step1} are finite. Therefore,
\[
\sup_{0\leq t\leq T}
\mathbb{E}
\left[
\|(\xi x)(t)\|^2
\right]
<\infty.
\]
Moreover, the strong continuity of $S(t)$, the continuity of the
deterministic convolution, and the mean-square continuity of the
fractional stochastic convolution imply that $\xi x\in
C\left([0,T];L^2(\Omega;\mathcal{H})\right)$. Thus, $\xi(\mathcal{X}_T)\subset\mathcal{X}_T$.

\textbf{Step 2: $\xi$ is a contraction for a suitable equivalent norm.}

Let $x,y\in\mathcal{X}_T$. Since the stochastic convolution is additive,
it cancels when considering the difference
$(\xi x)(t)-(\xi y)(t)$. Thus,
\[
(\xi x)(t)-(\xi y)(t)
=
\int_0^t
S(t-s)
\left(
F(x(s))-F(y(s))
\right)ds.
\]
Using the boundedness of $S(t)$, the global Lipschitz condition
\eqref{ass_non_lin_lipschitz} and the definition of the weighted norm $\|x(s)-y(s)\|_{L^2(\Omega;\mathcal{H})}\leq e^{us}\|x-y\|_{\mathcal{X}_T,u}$, we have
\begin{align*}
\|(\xi x)(t)-(\xi y)(t)\|_{L^2(\Omega;\mathcal{H})}
&\leq\int_0^t\|S(t-s)\left(F(x(s))-F(y(s))\right)\|_{L^2(\Omega;\mathcal{H})}ds\nonumber\\
&\leq CL\int_0^t\|x(s)-y(s)\|_{L^2(\Omega;\mathcal{H})}ds\nonumber\\
&\leq CL\left(\int_0^te^{us}ds\right)\|x-y\|_{\mathcal{X}_T,u}\nonumber\\
&=\frac{CL}{u}\left(e^{ut}-1\right)\|x-y\|_{\mathcal{X}_T,u}.
\end{align*}
Multiplying by $e^{-ut}$ and taking the supremum over $t\in[0,T]$, we obtain
\begin{align*}
\|\xi x-\xi y\|_{\mathcal{X}_T,u}
&\leq \frac{CL}{u}\left(1-e^{-uT}\right)\|x-y\|_{\mathcal{X}_T,u}\nonumber\\
&\leq\frac{CL}{u}\|x-y\|_{\mathcal{X}_T,u}.
\end{align*}
We choose $u>CL$. Then $\frac{CL}{u}<1$ and therefore $\xi$ is a contraction on $\mathcal{X}_T$ with respect to the equivalent norm $\|\cdot\|_{\mathcal{X}_T,u}$. By the Banach fixed-point theorem, there exists a unique $X\in\mathcal{X}_T$ such that $\xi X=X$. Consequently, $X$ is the unique mild solution of \eqref{pb}. Finally, applying the estimates obtained in Step 1 to the fixed point
$X$ and using the linear growth of $F$, followed by Gronwall's
inequality if necessary, yields
\[
\sup_{0\leq t\leq T}\mathbb{E}\left[\|X(t)\|^2\right]
\leq
C_T\left(1+\mathbb{E}\|X_0\|^2+\|\phi\|_{L_2^0}^2+\left\|A^{\frac{\beta}{2}}\phi\right\|_{L_2^0}^2\right)<\infty.
\]

This completes the proof.
\end{proof}

\section{Regularity of the mild solution}
\label{reg}	
In all that follows, $C$ denotes a generic positive constant that may change from line to line. We also let $\varepsilon>0$ be sufficiently small whenever needed. The following result establishes the spatial and temporal regularity of the mild solution \eqref{mild}. These regularity estimates will be used in the strong convergence analysis.

\begin{thm}
\label{regfrac}
Let $H\in\left(0,\frac12\right)$ and let
$1-2H<\beta\leq 1$.
Under Assumptions \ref{initdata}--\ref{non}, the unique mild solution
$X$ of \eqref{pb}, given by \eqref{mild}, satisfies
\begin{equation}
\label{sparegfrac}
\sup_{0\leq t\leq T}
\left\|
A^{\frac{2H+\beta}{2}}X(t)
\right\|_{L^2(\Omega;\mathcal H)}
\leq
C
\left(
1+
\left\|
A^{\frac{2H+\beta}{2}}X_0
\right\|_{L^2(\Omega;\mathcal H)}
\right).
\end{equation}
Moreover, for every $0\leq t_1<t_2\leq T$,
\begin{equation}
\label{timeregfrac}
\left\|
X(t_2)-X(t_1)
\right\|_{L^2(\Omega;\mathcal H)}
\leq
C
(t_2-t_1)^H
\left(
1+
\left\|
A^{\frac{2H+\beta}{2}}X_0
\right\|_{L^2(\Omega;\mathcal H)}
\right).
\end{equation}
Analogous estimates hold for the semidiscrete mild solution $X^h$,
with $A$ replaced by $A_h$, provided that the corresponding discrete
semigroup estimates are uniform with respect to the spatial
discretization parameter $h$.
\end{thm}

\begin{proof}
We first prove the spatial regularity estimate
\eqref{sparegfrac}. Recall that
\[
X(t)=S(t)X_0+\int_0^t S(t-s)F(X(s))\,ds+\int_0^t S(t-s)\phi\,dB_Q^H(s).
\]
Applying $A^{\frac{2H+\beta}{2}}$ to the mild formulation and using
the triangle inequality, followed by
$(a+b+c)^2\leq 3(a^2+b^2+c^2)$, we obtain
\begin{eqnarray}
	\label{space_split}
	\left\|A^{\frac{2H+\beta}2}X(t)\right\|^2_{L^2(\Omega;\mathcal{H})}&\leq& 3\left\|A^{\frac{2H+\beta}2}S(t)X_0\right\|^2_{L^2(\Omega;\mathcal{H})}+ 3\left\|\int_0^t A^{\frac{2H+\beta}2}S(t-s)F(X(s))ds\right\|^2_{L^2(\Omega;\mathcal{H})}\nonumber\\
	&+& 3\left\|\int_0^t A^{\frac{2H+\beta}2}S(t-s)\phi dB^H_Q(s) \right\|^2_{L^2(\Omega;\mathcal{H})}\nonumber\\
&=:& 3\sum_{i=1}^3 II_i(t).
\end{eqnarray}
Using the commutativity of $A^{\frac{2H+\beta}2}$ and $S(t)$ on
$\mathcal D(A^{\frac{2H+\beta}2})$ together with the boundedness of the analytic
semigroup, we have
\begin{equation}
\label{space_II1}
II_1(t)\leq C\left\|A^{\frac{2H+\beta}2} X_0\right\|_{L^2(\Omega;\mathcal H)}^2.
\end{equation}
For the deterministic convolution, using Proposition~\ref{semigroup}
with $\delta=\frac{2H+\beta}{2}$, together with the linear growth condition \eqref{ass_non_lin3} in Assumption~\ref{non} and the a priori bound provided by Theorem~\ref{well_poss_pbfrac}, we obtain
\begin{equation}
\label{space_II2}
II_2(t)\leq C.
\end{equation}
It remains to estimate the stochastic convolution $II_3(t)$. Using Lemma \ref{fracint} yields
\begin{align*}
II_3(t)
&:=
\left\|
\int_0^t
A^{\frac{2H+\beta}{2}}
S(t-s)\phi\,dB_Q^H(s)
\right\|_{L^2(\Omega;\mathcal H)}^2
\\
&=
\frac{H(1-2H)}{2}
\int_0^t\int_0^t
\left\|
A^{\frac{2H+\beta}{2}}
\left(
S(t-u)-S(t-v)
\right)\phi
\right\|_{L_2^0}^2
|u-v|^{2H-2}\,dv\,du
\\
&\quad+
H
\int_0^t
\left\|
A^{\frac{2H+\beta}{2}}
S(t-s)\phi
\right\|_{L_2^0}^2
\left(
s^{2H-1}+(t-s)^{2H-1}
\right)ds
\\
&=:II_{31}(t)+II_{32}(t).
\end{align*}
Let $\varepsilon>0$ be sufficiently small such that $0<\varepsilon<\min\{2H,1-\beta\}$ and $\varepsilon=2\varepsilon'$. For the estimate of $II_{31}$, inserting appropriate powers of $A$,
using Proposition \ref{semigroup} with $\delta=\frac{1-\varepsilon'}{2}$ and $\gamma=\frac{1-2H+\varepsilon'}{2}$,
Assumption \ref{noise} with $\kappa=\frac{\beta+\varepsilon}{2}$ and  \eqref{beta_identity}, we obtain
\begin{align}
\label{spa31}
II_{31}
&=\frac{H(1-2H)}{2}\int_0^t\int_0^t\left\|A^{\frac{2H+\beta}{2}}\left(S(t-u)-S(t-v)\right)\phi\right\|_{L_2^0}^2
|u-v|^{2H-2}\,dv\,du\nonumber\\
&=\frac{H(1-2H)}{2}\int_0^t\int_0^u\left\|A^{\frac{1-\epsilon'}2}S(t-u)\right\|_{L(\mathcal H)}^2\nonumber\\
&\qquad\times
\left\|
A^{\frac{2H-1-\varepsilon'}2}
\left(
I-S(u-v)
\right)
\right\|_{L(\mathcal H)}^2
\left\|
A^{\frac{\beta+\varepsilon}{2}}\phi
\right\|_{L_2^0}^2
(u-v)^{2H-2}\,dv\,du
\nonumber\\
&\quad+
\frac{H(1-2H)}{2}
\int_0^t\int_u^t
\left\|
A^{\frac{1-\varepsilon'}2}S(t-v)
\right\|_{L(\mathcal H)}^2
\nonumber\\
&\qquad\times
\left\|
A^{\frac{2H-1-\varepsilon'}2}
\left(
I-S(v-u)
\right)
\right\|_{L(\mathcal H)}^2
\left\|
A^{\frac{\beta+\varepsilon}{2}}\phi
\right\|_{L_2^0}^2
(v-u)^{2H-2}\,dv\,du
\nonumber\\
&\leq
C
\int_0^t\int_0^u
(t-u)^{-1+\varepsilon'}
(u-v)^{-1+\varepsilon'}
\,dv\,du
\\
&\hspace{2cm} +
C
\int_0^t\int_u^t
(t-v)^{-1+\varepsilon'}
(v-u)^{-1+\varepsilon'}
\,dv\,du
\nonumber\\
&\leq
C
\int_0^t
(t-u)^{-1+\varepsilon'}
u^{\varepsilon'}\,du+
C\mathcal B(\varepsilon',\varepsilon')
\int_0^t
(t-u)^{-1+2\varepsilon'}\,du
\nonumber\\
&=
C
\mathcal B(\varepsilon',1+\varepsilon')
t^{2\epsilon'}
+
C
\mathcal B(\varepsilon',\varepsilon')
\frac{t^{2\varepsilon'}}{2\varepsilon'}
\nonumber\\
&\leq
Ct^{\varepsilon}
\leq
C.
\end{align}
For the estimate of $II_{32}$, using Proposition \ref{semigroup} with $\delta=\frac{2H-\varepsilon}{2}$, Assumption \ref{noise} with $\kappa=\frac{\beta+\varepsilon}{2}$, and \ref{beta_identity}, we obtain
\begin{align}
\label{spa32}
II_{32}
&:=
H
\int_0^t
\left\|
A^{\frac{2H+\beta}{2}}
S(t-s)\phi
\right\|_{L_2^0}^2
\left(
s^{2H-1}+(t-s)^{2H-1}
\right)ds
\nonumber\\
&\leq
H
\int_0^t
\left\|
A^{\frac{2H-\varepsilon}{2}}
S(t-s)
\right\|_{L(\mathcal H)}^2
\left\|
A^{\frac{\beta+\varepsilon}{2}}\phi
\right\|_{L_2^0}^2
\left(
s^{2H-1}+(t-s)^{2H-1}
\right)ds
\nonumber\\
&\leq C\int_0^t(t-s)^{-2H+\varepsilon}s^{2H-1}\,ds+C\int_0^t(t-s)^{-2H+\varepsilon}(t-s)^{2H-1}\,ds
\nonumber\\
&=
C
\mathcal B(2H,1-2H+\varepsilon)
t^\epsilon
+
C
\int_0^t
(t-s)^{-1+\varepsilon}\,ds
\nonumber\\
&=
C
\mathcal B(2H,1-2H+\varepsilon)
t^\varepsilon
+
\frac{C}{\varepsilon}
t^\varepsilon
\nonumber\\
&\leq
Ct^\varepsilon
\leq
C.
\end{align}
Combining \eqref{space_split}, \eqref{space_II1},
\eqref{space_II2}, \eqref{spa31}, \eqref{spa32}, and taking the square root, we obtain
\[
\left\|A^{\frac{2H+\beta}{2}}X(t)\right\|_{L^2(\Omega;\mathcal H)}
\leq
C\left(
1+
\left\|A^{\frac{2H+\beta}{2}}X_0\right\|_{L^2(\Omega;\mathcal H)}
\right).
\]
Therefore,
\[
\sup_{0\leq t\leq T}
\left\|A^{\frac{2H+\beta}{2}}X(t)\right\|_{L^2(\Omega;\mathcal H)}
\leq
C\left(
1+
\left\|A^{\frac{2H+\beta}{2}}X_0\right\|_{L^2(\Omega;\mathcal H)}
\right).
\]
We now prove the temporal regularity estimate
\eqref{timeregfrac}. Let $0\leq t_1<t_2\leq T$, using the mild formulation and the semigroup property,
\begin{align*}
&X(t_2)-X(t_1)\\
&=\left(S(t_2-t_1)-I\right)X(t_1)+\int_{t_1}^{t_2}S(t_2-s)F(X(s))\,ds+\int_{t_1}^{t_2}S(t_2-s)\phi\,dB_Q^H(s).
\end{align*}
Consequently,
\[
\left\|
X(t_2)-X(t_1)
\right\|_{L^2(\Omega;\mathcal H)}
\leq
J_1+J_2+J_3,
\]
where the three terms correspond to the three terms on the
right-hand side. For $J_1$, using Proposition \ref{semigroup} with
$\gamma=\frac{2H+\beta}{2}$ and the spatial regularity estimate already proved, we have
\[
\begin{aligned}
    J_1
&=
\left\|
\left(
S(t_2-t_1)-I
\right)
X(t_1)
\right\|_{L^2(\Omega;\mathcal H)}
\nonumber\\
&\leq
C
(t_2-t_1)^\frac{2H+\beta}{2}
\left\|
A^\frac{2H+\beta}{2} X(t_1)
\right\|_{L^2(\Omega;\mathcal H)}
\nonumber\\
&\leq
C
(t_2-t_1)^\frac{2H+\beta}{2}
\left(
1+
\left\|
A^\frac{2H+\beta}{2} X_0
\right\|_{L^2(\Omega;\mathcal H)}
\right).
\label{J1}
\end{aligned}
\]
For $J_2$, using Proposition~\ref{semigroup} with $\delta=0$,
the linear growth condition \eqref{ass_non_lin3} in
Assumption~\ref{non}, and the a priori bound from
Theorem~\ref{well_poss_pbfrac}, we obtain
\[
\begin{aligned}
J_2
&\leq
C\int_{t_1}^{t_2}
\left(1+\|X(s)\|_{L^2(\Omega;\mathcal H)}\right)\,ds\\
&\leq C(t_2-t_1)
\leq C_T(t_2-t_1)^{\frac{2H+\beta}{2}}.
\end{aligned}
\]
Therefore,
\begin{equation}
\label{J12}
J_1+J_2
\leq
C
(t_2-t_1)^{\frac{2H+\beta}{2}}
\left(
1+
\left\|
A^{\frac{2H+\beta}{2}} X_0
\right\|_{L^2(\Omega;\mathcal H)}
\right).
\end{equation}

It remains to estimate $J_3$, using Lemma \ref{fracint} leads to
\[
\begin{aligned}
J_3^2
&=
\left\|
\int_{t_1}^{t_2}
S(t_2-s)\phi\,dB_Q^H(s)
\right\|_{L^2(\Omega;\mathcal H)}^2
\nonumber\\
&=
\frac{H(1-2H)}{2}
\int_{t_1}^{t_2}\int_{t_1}^{t_2}
\left\|
S(t_2-u)\phi-S(t_2-v)\phi
\right\|_{L_2^0}^2
|u-v|^{2H-2}\,dv\,du
\nonumber\\
&\quad+
H
\int_{t_1}^{t_2}
\left\|
S(t_2-s)\phi
\right\|_{L_2^0}^2
\left(
(s-t_1)^{2H-1}
+
(t_2-s)^{2H-1}
\right)ds
\nonumber\\
&=:J_{31}^2+J_{32}^2.
\end{aligned}
\]
For the estimate of $J_{31}^2$, we split the square $[t_1,t_2]^2$ into the two regions $v\leq u$ and $u\leq v$.
For $t_1\leq v\leq u\leq t_2$, using the semigroup property, $S(t_2-v)=S(t_2-u)S(u-v)$, and hence $S(t_2-u)-S(t_2-v)=S(t_2-u)\left(I-S(u-v)\right)$. Similarly, for $t_1\leq u\leq v\leq t_2$, $S(t_2-u)-S(t_2-v)=S(t_2-v)\left(S(v-u)-I\right)$. Therefore, inserting an appropriate power of $A$, using the
boundedness of the semigroup, Proposition \ref{semigroup} with
$\gamma=\frac{\beta}{2}$, and Assumption \ref{noise} with
$\kappa=\frac{\beta}{2}$, we obtain
\begin{align*}
&J_{31}^2\nonumber\\
&\leq C\int_{t_1}^{t_2}\int_{t_1}^{u}\left\|S(t_2-u)\right\|_{L(\mathcal H)}^2\left\|A^{-\frac{\beta}{2}}\left(I-S(u-v)\right)\right\|_{L(\mathcal H)}^2\left\|A^{\frac{\beta}{2}}\phi\right\|_{L_2^0}^2(u-v)^{2H-2}\,dv\,du\nonumber\\
&+C\int_{t_1}^{t_2}\int_{u}^{t_2}\left\|S(t_2-v)\right\|_{L(\mathcal H)}^2\left\|A^{-\frac{\beta}{2}}\left(I-S(v-u)\right)\right\|_{L(\mathcal H)}^2\left\|A^{\frac{\beta}{2}}\phi\right\|_{L_2^0}^2(v-u)^{2H-2}\,dv\,du
\nonumber\\
&\leq
C\int_{t_1}^{t_2}\int_{t_1}^{u}(u-v)^{2H+\beta-2}\,dv\,du+C\int_{t_1}^{t_2}\int_{u}^{t_2}(v-u)^{2H+\beta-2}\,dv\,du.
\end{align*}
Since $2H+\beta-1>0$ by the assumption $\beta>1-2H$, both integrals are finite. More
precisely,
\begin{align}
J_{31}^2
&\leq
\frac{C}{2H+\beta-1}\int_{t_1}^{t_2}(u-t_1)^{2H+\beta-1}\,du+\frac{C}{2H+\beta-1}\int_{t_1}^{t_2}(t_2-u)^{2H+\beta-1}\,du\nonumber\\
&\leq C(t_2-t_1)^{2H+\beta}.
\label{J31_final}
\end{align}
For the estimate of $J_{32}^2$, using the boundedness of the
semigroup $S(t)$ and Assumption \ref{noise} with $\kappa=0$, we obtain
\begin{align}
\label{J32}
J_{32}^2
&:=
H\int_{t_1}^{t_2}\left\|S(t_2-s)\phi\right\|_{L_2^0}^2\left((s-t_1)^{2H-1}+(t_2-s)^{2H-1}\right)ds\nonumber\\
&\leq
C\int_{t_1}^{t_2}\left\|S(t_2-s)\right\|_{L(\mathcal H)}^2\left\|\phi\right\|_{L_2^0}^2\left((s-t_1)^{2H-1}+(t_2-s)^{2H-1}\right)ds
\nonumber\\
&\leq
C\int_{t_1}^{t_2}\left((s-t_1)^{2H-1}+(t_2-s)^{2H-1}\right)ds\nonumber\\
&\leq C(t_2-t_1)^{2H}.
\end{align}
Therefore, adding \eqref{J31_final} and \eqref{J32}, and then taking the square root, yields
\begin{align}
J_3
&\leq
C(t_2-t_1)^{H+\frac{\beta}{2}}
+
C(t_2-t_1)^H
\nonumber\\
&\leq
C(t_2-t_1)^H,
\label{J3_final}
\end{align}
Combining \eqref{J3_final} with \eqref{J12}, we get
\begin{align*}
\left\|
X(t_2)-X(t_1)
\right\|_{L^2(\Omega;\mathcal H)}
&\leq
C(t_2-t_1)^{\frac{2H+\beta}{2}}
+
C(t_2-t_1)^H
\\
&\leq
C(t_2-t_1)^H.
\end{align*}

The proof of the temporal regularity is thus completed. The proof for the semidiscrete solution $X^h$ follows analogously by using the corresponding estimates for $S_h(t)$ and $A_h$, with
constants independent of $h$.
\end{proof}

\section{Space approximation and error estimates}
\label{spa_conv}

We consider the discretization of the spatial domain by a finite element triangulation with maximal length $h$ satisfying the usual regularity assumptions. Let $V_h\subset V$ denotes the space of continuous functions that are piecewise linear over triangulation $J_h$. To discretise in space, we introduce $P_h$ from $L^2(\Omega)$ to $V_h$ define for $u\in L^2(\Omega)$ by
\begin{eqnarray*}
	\label{pro}
	(P_h u,\xi)=(u,\xi),\hspace{2cm}\forall \xi\in V_h.
\end{eqnarray*}
The discrete operator $A_h: V_h\rightarrow V_h$ is defined by
\begin{eqnarray*}
	\label{dis}
	(A_h \rho,\xi)=-a(\rho,\xi),\hspace{2cm}\forall \rho, \xi\in V_h,
\end{eqnarray*}
where $a$ is the corresponding bilinear form of $A$. Like the operator $A$, the discrete operator $-A_h$ is also the generator of an analytic semigroup $S_h(t):=e^{-tA_h}$. The semidiscrete space version of problem \eqref{pb} is to find $X^h(t)=X^h(\cdot,t)$ such that for $t\in[0,T]$
\begin{eqnarray}
	\label{dispb}
	\left\{
	\begin{array}{ll}
	dX^h(t)+A_hX^h(t)dt=P_hF(X^h(t))dt+P_h\phi dB^h_Q(t), \\
	X^h(0)=P_hX_0,\,\,\,\,t\in[0,T].
	\end{array}
	\right.
\end{eqnarray}
Note that $A_h$, $P_hF$ and $P_h\phi$ satisfy the same assumptions as $A$, $F$, $\phi$ respectively. The mild solution of \eqref{dispb} can be represented as follows
\begin{eqnarray}
	\label{dissol}
	X^h(t)=S_h(t)X^h_0+\int_{0}^{t}S_h(t-s)P_hF(X^h(s))ds+\int_{0}^{t}S_h(t-s)P_h\phi dB^H_Q(s),
\end{eqnarray}
where $S_h(t):=e^{-tA_h}$ is the semi discrete version of $S(t)$. Let us define the error operator $T_h(t):=S(t)-S_h(t)P_h$,
then we have the following lemma.
\begin{lem}(\cite[Lemmas 3.1 and 3.2(ii)]{Mukb})
\label{disop} 
For $r\in[0,2]$, $\rho\leq r$, $t\in(0,T]$, and
$v\in\mathcal D(A^{\rho/2})$, there exists a positive constant $C$
such that
\begin{equation}
\label{nordisop}
\|T_h(t)v\|
\leq
C h^r t^{-\frac{r-\rho}{2}}
\|v\|_{\rho}.
\end{equation}

Moreover, for $r\in[1,2]$, $1\leq\rho\leq r$, and
$v\in\mathcal D(A^{\rho/2})$,
\begin{equation}
\label{derivative_error_operator}
\left\|
\frac{d}{dt}T_h(t)v
\right\|
\leq
C h^\rho t^{-1}\|v\|_{\rho}
+
C h^r t^{-1-\frac{r-\rho}{2}}\|v\|_{\rho}.
\end{equation}
Here, $C$ is independent of $h$ and $t$.
\end{lem}

We are now in a position to prove the following theorem, which provides
a strong mean-square estimate for the error between the mild solution
of \eqref{pb} and its spatial semidiscrete approximation.
\begin{thm}[Spatial error estimate]
\label{spaconv}
Let $X$ and $X^h$ be the mild solutions of \eqref{pb} and
\eqref{dispb}, respectively. Suppose that Assumptions
\ref{initdata}--\ref{non} hold. Then,
\begin{equation}
\left\|X(t)-X^h(t)\right\|_{L^2(\Omega;\mathcal H)}
\leq
C h^{2H+\beta},
\qquad 0\leq t\leq T.
\end{equation}
\end{thm}
\begin{proof}
Define $e(t):=X(t)-X^h(t)$. By \eqref{mild} and \eqref{dissol}, taking the norm and using the triangle inequality, we deduce
\begin{eqnarray*}
&&\|e(t)\|_{L^2(\Omega;\mathcal H)}\\
&\leq&
\left\|
S(t)X_0
+\int_0^t S(t-s)F(X(s))ds
+\int_0^t S(t-s)\phi\,dB_Q^H(s)
\right.
\\
&&\left.
\hspace{0.5cm}
-S_h(t)X_0^h
-\int_0^t S_h(t-s)P_hF(X^h(s))ds
-\int_0^t S_h(t-s)P_h\phi\,dB_Q^H(s)
\right\|_{L^2(\Omega;\mathcal H)}
\\
&\leq&
\|T_h(t)X_0\|_{L^2(\Omega;\mathcal H)}+\left\|\int_0^t\left[S(t-s)F(X(s))-S_h(ts)P_hF(X^h(s))\right]ds\right\|_{L^2(\Omega;\mathcal H)}
\\
&&+
\left\|
\int_0^t
T_h(t-s)\phi\,dB_Q^H(s)
\right\|_{L^2(\Omega;\mathcal H)}
\\
&=:&
\sum_{i=1}^3K_i.
\end{eqnarray*}
The estimates of $K_1$ and $K_2$ follow the same lines as
\cite[(67)--(69)]{Nou}. Using Lemma \ref{disop} with
$r=\rho=2H+\beta$ for $K_1$, and with $r=2H+\beta$, $\rho=0$
for $K_2$, together with Assumptions \ref{initdata}, \ref{non}
and Theorem \ref{regfrac}, we obtain
\begin{equation}
\label{spaerr12}
K_1+K_2
\leq
C h^{2H+\beta}
+
C\int_0^t
\|e(s)\|_{L^2(\Omega;\mathcal H)}\,ds.
\end{equation}
We now estimate $K_3$. Applying Lemma \ref{fracint}, we obtain
\begin{eqnarray*}
K_3^2
&=&
\mathbb E
\left\|
\int_0^t
T_h(t-s)\phi\,dB_Q^H(s)
\right\|^2
\\
&=&
\frac{H(1-2H)}{2}
\int_0^t\int_0^t
\left\|
T_h(t-u)\phi
-
T_h(t-v)\phi
\right\|_{L_2^0}^2
|u-v|^{2H-2}\,du\,dv
\\
&&+
H\int_0^t
\left\|
T_h(t-s)\phi
\right\|_{L_2^0}^2
\left(
s^{2H-1}
+
(t-s)^{2H-1}
\right)ds
\\
&=:&
K^2_{31}+K^2_{32}.
\end{eqnarray*}
We first estimate $K_{31}$. 
For $0\leq u\leq v<t$, using \eqref{derivative_error_operator} with $r=\rho=2H+\beta$, the estimate $\ln(1+x)\leq x$, $x\geq 0$, and Assumption \ref{noise} with $\kappa = \frac{2H+\beta}2$, we obtain
\begin{align}
\left\|
T_h(t-u)\phi
-
T_h(t-v)\phi
\right\|_{L_2^0}&\leq
\int_{t-v}^{t-u}
\left\|
\frac{d}{d\tau}T_h(\tau)\phi
\right\|_{L_2^0}d\tau
\nonumber\\
&\leq
Ch^{2H+\beta}\left\|
A^{\frac {2H+\beta}2}\phi\right\|_{L_2^0}
\int_{t-v}^{t-u}\tau^{-1}d\tau
\nonumber\\
&\leq
Ch^{2H+\beta} \ln\left(\frac{t-u}{t-v}\right)
\nonumber\\
&\leq
Ch^{2H+\beta}
\frac{(v-u)}{t-v}.
\label{Th_diff_bound1}
\end{align}
On the other hand, using \eqref{nordisop} with $r=\rho=2H+\beta$, and Assumption \ref{noise} with $\kappa = \frac{2H+\beta}2$ yields
\begin{align}
\left\|
T_h(t-u)\phi-T_h(t-v)\phi\right\|_{L_2^0}&\leq
\left\|T_h(t-u)\phi\right\|_{L_2^0}+\left\|T_h(t-v)\phi\right\|_{L_2^0}\nonumber\\
&\leq
Ch^{2H+\beta}\left\|A^{\frac {2H+\beta}2}\phi\right\|_{L_2^0}\leq
Ch^{2H+\beta}.
\label{Th_diff_bound2}
\end{align}
Combining \eqref{Th_diff_bound1} and \eqref{Th_diff_bound2}, and
using the estimate $\min\{1,x\}\leq x^\alpha$, $x\geq0$, $0<\alpha<1$, we obtain
\begin{align*}
\left\|
T_h(t-u)\phi
-
T_h(t-v)\phi
\right\|_{L_2^0}
&\leq
Ch^{2H+\beta}
\min\{1,\frac{(v-u)}{t-v}\}\nonumber\\
&\leq
Ch^{2H+\beta}
(v-u)^\alpha
(t-v)^{-\alpha}.
\end{align*}
Choose $\frac12-H<\alpha<\frac12$, then $2\alpha+2H-2>-1$. Therefore, using the symmetry of the double integral,
\begin{align}
K^2_{31}
&\leq
Ch^{2(2H+\beta)}
\int_0^t
\int_0^v
(v-u)^{2\alpha}
(t-v)^{-2\alpha}
(v-u)^{2H-2}
\,du\,dv
\nonumber\\
&=
Ch^{2(2H+\beta)}
\int_0^t
(t-v)^{-2\alpha}
\left(
\int_0^v
(v-u)^{2\alpha+2H-2}du
\right)dv\nonumber\\
&\leq
Ch^{2(2H+\beta)}
\int_0^t
v^{2\alpha+2H-1}
(t-v)^{-2\alpha}dv
\nonumber\\
&=
Ch^{2(2H+\beta)}
t^{2H}
\mathcal B
\left(
2\alpha+2H,
1-2\alpha
\right)
\nonumber\\
&\leq
Ch^{2(2H+\beta)}.
\label{spafraer1}
\end{align}
For the term $K_{32}$, applying Lemma \ref{disop} with
$r=2H+\beta$ for both terms, $\rho=\beta$ for the first term and
$\rho=\beta+\varepsilon$ for the second one, together with Assumption
\ref{noise} with $\kappa=\frac{\beta}{2}$ and $\kappa=\frac{\beta+\varepsilon}{2}$,
respectively, we obtain
\begin{eqnarray}
\label{spafraer2}
K^2_{32}
&:=&
H\int_0^t
\left\|
T_h(t-s)\phi
\right\|_{L_2^0}^2
\left(
s^{2H-1}
+
(t-s)^{2H-1}
\right)ds
\nonumber\\
&=&
H\int_0^t
\left\|
T_h(t-s)\phi
\right\|_{L_2^0}^2
s^{2H-1}ds+H\int_0^t\left\|T_h(t-s)\phi\right\|_{L_2^0}^2(t-s)^{2H-1}ds\nonumber\\
&\leq&
C_H
h^{2(2H+\beta)}
\int_0^t
(t-s)^{-2H}
\left\|
A^{\frac{\beta}{2}}\phi
\right\|_{L_2^0}^2
s^{2H-1}ds
\nonumber\\
&&+
C_H
h^{2(2H+\beta)}
\int_0^t
(t-s)^{-2H+\epsilon}
\left\|
A^{\frac{\beta+\varepsilon}{2}}\phi
\right\|_{L_2^0}^2
(t-s)^{2H-1}ds
\nonumber\\
&\leq&
C_Hh^{2(2H+\beta)}
\int_0^t
(t-s)^{-2H}s^{2H-1}ds+C_Hh^{2(2H+\beta)}\int_0^t(t-s)^{-1+\varepsilon}ds\nonumber\\
&\leq&
C_Hh^{2(2H+\beta)}
\mathcal B(1-2H,2H)
+
C_Hh^{2(2H+\beta)}
\frac{t^\varepsilon}{\varepsilon}
\nonumber\\
&\leq&
Ch^{2(2H+\beta)}.
\end{eqnarray}
Here, $\varepsilon>0$ is chosen sufficiently small such that the
corresponding noise regularity in Assumption \ref{noise} is
satisfied. Combining \eqref{spafraer1} and \eqref{spafraer2}, and taking the square root yields
\begin{equation}
\label{spaerr3}
K_3
\leq
Ch^{2H+\beta}.
\end{equation}

Finally, combining \eqref{spaerr12} and \eqref{spaerr3}, we obtain
\begin{eqnarray*}
\|e(t)\|_{L^2(\Omega;\mathcal H)}
&\leq&
Ch^{2H+\beta}
+
C\int_0^t
\|e(s)\|_{L^2(\Omega;\mathcal H)}ds.
\end{eqnarray*}
Applying the continuous Gronwall inequality yields
\[
\|e(t)\|_{L^2(\Omega;\mathcal H)}
\leq
Ch^{2H+\beta},
\qquad
0\leq t\leq T.
\]
The proof is thus completed.
\end{proof}

\section{Fully discrete stochastic exponential scheme and strong error estimates}
\label{schemes}

In this section, we introduce a fully discrete stochastic exponential
approximation of the semidiscrete problem \eqref{dispb} and establish
its strong convergence rate. Throughout this section, let $0=t_0<t_1<\cdots<t_M=T$, $t_m=m\Delta t$, $\Delta t=\frac{T}{M}$. For every $s\in[0,T)$, we define
\[
[s]:=t_i \qquad\text{and}\qquad [s]^m:=i, \qquad s\in[t_i,t_{i+1}).
\]
Recall that the semidiscrete mild solution is given by
\begin{equation}
\label{soldis}
X^h(t)=
S_h(t)X_0^h+\int_0^tS_h(t-s)P_hF(X^h(s))\,ds +\int_0^tS_h(t-s)P_h\phi\,dB_Q^H(s),\quad t\in[0,T].
\end{equation}
We consider the following stochastic exponential time discretization:
\begin{equation}
\label{SETD1}
X_m^h=
S_h(t_m)X_0^h+\int_0^{t_m}S_h(t_m-s)P_hF(X_{[s]^m}^h)\,ds+\int_0^{t_m}S_h(t_m-[s])P_h\phi\,dB_Q^H(s).
\end{equation}
Using the semigroup property, the scheme can equivalently be written
in the one-step form
\begin{equation*}
X_{m+1}^h=S_h(\Delta t)X_m^h+\int_0^{\Delta t}S_h(\Delta t-s)P_hF(X_m^h)\,ds+S_h(\Delta t)P_h\phi\Delta B_{Q,m}^H,
\end{equation*}
where $\Delta B_{Q,m}^H=B_Q^H(t_{m+1})-B_Q^H(t_m)$.
Equivalently,
\begin{equation}
\label{SETD1_phi}
X_{m+1}^h
=S_h(\Delta t)X_m^h+\Delta t\,\varphi_1(\Delta t A_h)P_hF(X_m^h)+S_h(\Delta t)P_h\phi\Delta B_{Q,m}^H,
\end{equation}
where $\displaystyle \varphi_1(z)=\frac{e^z-1}{z}$. The representation \eqref{SETD1_phi} corresponds to the SETDM1
scheme used in the numerical experiments.

\subsection{Auxiliary estimates}
\label{auxi}

We first establish an auxiliary estimate for the quadrature error
associated with the stochastic convolution.

\begin{lem}[Quadrature error of the stochastic convolution]
\label{stoch_quad_error}

Under Assumption \ref{noise}, for every $m=1,\ldots,M$,
\begin{equation}
\label{stoch_quad_final}
\left\|
\int_0^{t_m}
\left[
S_h(t_m-s)-S_h(t_m-[s])
\right]
P_h\phi\,dB_Q^H(s)
\right\|_{L^2(\Omega;\mathcal H)}
\leq
C\Delta t^{\frac{2H+\beta}{2}},
\end{equation}
where $C$ is independent of $h$, $\Delta t$, and $m$.
\end{lem}

\begin{proof}

Define $\displaystyle G_m(s)=\left[S_h(t_m-s)-S_h(t_m-[s])\right]P_h\phi$, $s\in[0,t_m]$.\\
Using Lemma \ref{fracint}, we obtain
\begin{align}
&\left\|\int_0^{t_m}G_m(s)\,dB_Q^H(s)\right\|_{L^2(\Omega;\mathcal H)}^2\nonumber\\
&\leq
\frac{H(1-2H)}{2}
\int_0^{t_m}\int_0^{t_m}
\left\|
G_m(u)-G_m(v)
\right\|_{L_2^0}^2
|u-v|^{2H-2}\,dv\,du
\nonumber\\
&\quad+
H
\int_0^{t_m}
\left\|
G_m(s)
\right\|_{L_2^0}^2
\left(
s^{2H-1}
+
(t_m-s)^{2H-1}
\right)ds
\nonumber\\
&=:L_{21}^2+L_{22}^2.
\label{L2split}
\end{align}
We estimate $L_{21}$ and $L_{22}$ separately. Let $I_i=[t_i,t_{i+1})$, $i=0,\ldots,m-1$.
For $s\in I_i$, we have $[s]=t_i$. Using the semigroup property, Proposition \ref{semigroup} with
$\gamma=\frac{1+\beta}{2}$, the boundedness of $S_h(t)$, the stability of the finite element
projection, and Assumption \ref{noise}, we obtain
\begin{align}
\left\|
G_m(s)
\right\|_{L_2^0}
&\leq
\left\|
S_h(t_m-s)
\right\|_{L(\mathcal H)}
\left\|
A_h^{-\frac{1+\beta}{2}}
\left(
I-S_h(s-t_i)
\right)
\right\|_{L(\mathcal H)}
\left\|
A_h^{\frac{1+\beta}{2}} P_h\phi
\right\|_{L_2^0}
\nonumber\\
&\leq
C(s-t_i)^{\frac{1+\beta}{2}}
\nonumber\\
&\leq
C\Delta t^{\frac{1+\beta}{2}}.
\label{G_pointwise_kappa}
\end{align}
We now decompose $L_{21}^2$ as 
\begin{align*}
L_{21}^2
&=
\frac{H(1-2H)}{2}
\sum_{i=0}^{m-1}
\int_{I_i}\int_{I_i}
\left\|
G_m(u)-G_m(v)
\right\|_{L_2^0}^2
|u-v|^{2H-2}\,dv\,du, \\
&\qquad + \frac{H(1-2H)}{2}
\sum_{\substack{i,j=0\\i\neq j}}^{m-1}
\int_{I_i}\int_{I_j}
\left\|
G_m(u)-G_m(v)
\right\|_{L_2^0}^2
|u-v|^{2H-2}\,dv\,du,\\ 
&=: L_{21,\mathrm{diag}}^2+L_{21,\mathrm{off}}^2. 
\end{align*}
For the estimate of $L_{21,\mathrm{diag}}^2$, let $u,v\in I_i$
and assume without loss of generality that $v\leq u$. Since
$[u]=[v]=t_i$, the semigroup property, Proposition \ref{semigroup} with
$\gamma=\frac{1+\beta}{2}$, the boundedness of the
semigroup, the stability of the finite element projection, and
Assumption \ref{noise} yield
\begin{align}
L_{21,\mathrm{diag}}^2 &:=\frac{H(1-2H)}{2}
\sum_{i=0}^{m-1}
\int_{I_i}\int_{I_i}
\left\|
G_m(u)-G_m(v)
\right\|_{L_2^0}^2
|u-v|^{2H-2}\,dv\,du,\nonumber\\
&\leq \frac{H(1-2H)}{2}
\sum_{i=0}^{m-1}
\int_{I_i}\int_{I_i}
\left\|
S_h(t_m-u)
\right\|^2_{L(\mathcal H)}
\left\|
A_h^{-\frac{1+\beta}{2}}
\left(
I-S_h(u-v)
\right)
\right\|^2_{L(\mathcal H)}\nonumber\\
&\hspace{3cm}\times\left\|A_h^{\frac{1+\beta}{2}} P_h\phi\right\|^2_{L_2^0}
|u-v|^{2H-2}\,dv\,du,\nonumber\\
&\leq
C
\sum_{i=0}^{m-1}
\int_{I_i}\int_{I_i}
|u-v|^{2H+\beta-1}\,dv\,du,\nonumber\\
&\leq
Cm\Delta t^{2H+\beta+1} \leq C\Delta t^{2H+\beta}.
\label{L21_diag_final}
\end{align}

For the estimate of $L_{21,\mathrm{off}}^2$, let $u\in I_i$ and
$v\in I_j$, with $i\neq j$. From \eqref{G_pointwise_kappa},
\begin{align*}
\left\|
G_m(u)-G_m(v)
\right\|_{L_2^0}
&\leq
\left\|
G_m(u)
\right\|_{L_2^0}
+
\left\|
G_m(v)
\right\|_{L_2^0}
\nonumber\\
&\leq
C\Delta t^{\frac{1+\beta}{2}}.
\end{align*}
Hence
\begin{align}
L_{21,\mathrm{off}}^2
&\leq
C\Delta t^{1+\beta}
\sum_{\substack{i,j=0\\i\neq j}}^{m-1}
\int_{I_i}\int_{I_j}
|u-v|^{2H-2}\,dv\,du.
\label{L21off_step1}
\end{align}
For $j=i+k$, $k\geq1$, using the changes of variables $u=t_i+\Delta t\,x$, $v=t_{i+k}+\Delta t\,y$, $x,y\in[0,1]$,
we have
\begin{align*}
\int_{I_i}\int_{I_{i+k}}
|u-v|^{2H-2}\,dv\,du
&=
\Delta t^{2H}\rho_H(k),
\end{align*}
where $\displaystyle \rho_H(k) = \int_0^1\int_0^1 |k+y-x|^{2H-2}\,dy\,dx$.
For $k=1$, setting $a=1-x$ gives \\$\displaystyle \rho_H(1)=\int_0^1\int_0^1(a+y)^{2H-2}\,dy\,da<\infty$, since $2H-2>-2$. For $k\geq2$, $\rho_H(k)\leq C(k-1)^{2H-2}$.
Since $H<1/2$, we have $2H-2<-1$, therefore $\displaystyle \sum_{k=1}^{\infty}\rho_H(k)<\infty$, and since for each $k=1,\ldots,m-1$ there are exactly $m-k$ pairs of intervals $(I_i,I_{i+k})$, it follows from \eqref{L21off_step1} that
\begin{align}
L_{21,\mathrm{off}}^2
&\leq
C\Delta t^{1+\beta}
\sum_{\substack{i,j=0\\i\neq j}}^{m-1}
\int_{I_i}\int_{I_j}
|u-v|^{2H-2}\,dv\,du\nonumber\\
& \leq C\Delta t^{1+\beta}\Delta t^{2H}
\sum_{k=1}^{m-1}(m-k)\rho_H(k)
\nonumber\\
&\leq
Cm\Delta t^{2H+\beta+1}\sum_{k=1}^{\infty}\rho_H(k)\leq
C\Delta t^{2H+\beta}.
\label{L21_off_final}
\end{align}
Therefore, combining \eqref{L21_diag_final} and
\eqref{L21_off_final}, we obtain
\begin{equation}
\label{L21_final}
L_{21}^2
\leq
C\Delta t^{2H+\beta},
\end{equation}

We now estimate the second contribution in \eqref{L2split}.
From \eqref{G_pointwise_kappa}, we have
\begin{align}
L_{22}^2
&=
H
\int_0^{t_m}
\|G_m(s)\|_{L_2^0}^2
\left[
s^{2H-1}
+
(t_m-s)^{2H-1}
\right]ds
\nonumber\\
&\leq
C
\Delta t^{1+\beta}
\int_0^{t_m}
\left[
s^{2H-1}
+
(t_m-s)^{2H-1}
\right]ds
\nonumber\\
&\leq
C
\Delta t^{1+\beta}
t_m^{2H}
\nonumber\\
&\leq C\Delta t^{1+\beta}\leq C
\Delta t^{2H+\beta}.
\label{L22_final}
\end{align}
Finally, combining \eqref{L21_final}, \eqref{L22_final} and taking the square root yields
\[
\left\|
\int_0^{t_m}
\left[
S_h(t_m-s)-S_h(t_m-[s])
\right]
P_h\phi\,dB_Q^H(s)
\right\|_{L^2(\Omega;\mathcal H)}
\leq
C
\Delta t^{\frac{2H+\beta}{2}}.
\]
This proves \eqref{stoch_quad_final}.
\end{proof}

Before turning to the accumulated nonlinear drift defect, we establish
a uniform estimate for the Malliavin derivative of the semidiscrete
solution. This estimate will be used to control the random coefficients
arising in the nonlinear stochastic contribution.

\begin{lem}[Uniform estimate for the Malliavin derivative]
\label{lem_malliavin_Xh}
Under Assumption \ref{non}, for every $t\in[0,T]$,
$X^h(t)$ is Malliavin differentiable and, for
$0\leq\theta\leq t$,
\begin{equation}
\label{DXh_equation}
D_\theta X^h(t)
=
S_h(t-\theta)P_h\phi
+
\int_\theta^t
S_h(t-r)P_hF'(X^h(r))
D_\theta X^h(r)\,dr.
\end{equation}
Moreover, $D_\theta X^h(t)=0$, $\theta>t$, and
\begin{equation}
\label{DXh_canonical_bound}
\sup_{0\leq t\leq T}
\left\|
DX^h(t)
\right\|_{L^4(\Omega;\mathfrak H_H(L_2^0))}
\leq C,
\end{equation}
where $C$ is independent of $h$.
\end{lem}

\begin{proof}
The semidiscrete mild solution is given by
\begin{align}
X^h(t)=S_h(t)P_hX_0+\int_0^tS_h(t-r)P_hF(X^h(r))\,dr+\int_0^tS_h(t-r)P_h\phi\,dB_Q^H(r).
\label{Xh_mild_D}
\end{align}
Since $X_0$ is independent of $B_Q^H$, we have $D_\theta X_0=0$. Furthermore, since the integrand of the stochastic convolution is deterministic,
\[
D_\theta
\left(
\int_0^t
S_h(t-r)P_h\phi\,dB_Q^H(r)
\right)
=
\mathbf 1_{[0,t]}(\theta)
S_h(t-\theta)P_h\phi.
\]
Using the Malliavin chain rule, $D_\theta F(X^h(r))=F'(X^h(r))D_\theta X^h(r)$,
which yields \eqref{DXh_equation}. We now estimate $DX^h(t)$ in the canonical space
$\mathfrak H_H(L_2^0)$. Writing \eqref{DXh_equation} as a function
of $\theta$, we obtain
\begin{align*}
DX^h(t)=\mathbf 1_{[0,t]}(\cdot)S_h(t-\cdot)P_h\phi+\int_0^t\mathbf 1_{[0,r]}(\cdot)
S_h(t-r)P_hF'(X^h(r))DX^h(r)\,dr.
\end{align*}
Taking the norm in
$L^4(\Omega;\mathfrak H_H(L_2^0))$ and using Minkowski's
inequality yields
\begin{align}
\left\|
DX^h(t)
\right\|_{L^4(\Omega;\mathfrak H_H(L_2^0))}
&\leq
\left\|
\mathbf 1_{[0,t]}(\cdot)
S_h(t-\cdot)P_h\phi
\right\|_{\mathfrak H_H(L_2^0)}
\nonumber\\
&+
\int_0^t
\Big\|
\mathbf 1_{[0,r]}(\cdot)
S_h(t-r)P_hF'(X^h(r))
DX^h(r)
\Big\|_{L^4(\Omega;\mathfrak H_H(L_2^0))}
\,dr.
\label{DXh_norm_step1}
\end{align}

By the estimate of the stochastic convolution established previously,
together with Assumption \ref{noise} and the stability of the finite
element projection,
\begin{equation}
\label{DXh_noise_term}
\left\|
\mathbf 1_{[0,t]}(\cdot)
S_h(t-\cdot)P_h\phi
\right\|_{\mathfrak H_H(L_2^0)}
\leq C.
\end{equation}

Moreover, since $S_h(t)$ is uniformly bounded and
$F'\in C_b(\mathcal H;\mathcal L(\mathcal H))$, we have
\begin{align}
&
\Big\|
\mathbf 1_{[0,r]}(\cdot)
S_h(t-r)P_hF'(X^h(r))
DX^h(r)
\Big\|_{L^4(\Omega;\mathfrak H_H(L_2^0))}
\nonumber\\
&\qquad\leq
C
\left\|
DX^h(r)
\right\|_{L^4(\Omega;\mathfrak H_H(L_2^0))}.
\label{DXh_nonlinear_term}
\end{align}
Indeed, $S_h(t-r)P_hF'(X^h(r))$ does not depend on the Malliavin
variable $\theta$ and acts as a uniformly bounded operator on the
spatial component of $\mathfrak H_H(L_2^0)$.

Combining \eqref{DXh_norm_step1}, \eqref{DXh_noise_term} and \eqref{DXh_nonlinear_term}, we get
\[
\left\|
DX^h(t)
\right\|_{L^4(\Omega;\mathfrak H_H(L_2^0))}
\leq
C
+
C
\int_0^t
\left\|
DX^h(r)
\right\|_{L^4(\Omega;\mathfrak H_H(L_2^0))}
\,dr.
\]
An application of Gronwall's inequality then yields
\eqref{DXh_canonical_bound}.
\end{proof}


We next establish an estimate for the fractional covariance
norm of functions defined cellwise on the temporal grid. This result
will be used to control the stochastic contribution to the accumulated
nonlinear drift defect.

\begin{lem}[Cellwise fractional covariance estimate]
\label{lem_cellwise_covariance}
Let $H\in\left(0,\frac12\right)$, and consider the uniform partition $0=t_0<t_1<\cdots<t_M=T$, $t_i=i\Delta t$, $\Delta t=\frac{T}{M}$. Set $I_i=[t_i,t_{i+1})$, $i=0,\ldots,M-1$. Let $V$ be a separable Hilbert space and, for
$m=1,\ldots,M$, let $g:[0,t_m]\longrightarrow V$ be piecewise continuous. Assume that, for some constant $C>0$,
\begin{equation}
\label{cell_g_bound_new}
\|g(r)\|_V
\leq
C\Delta t,
\qquad
0\leq r\leq t_m,
\end{equation}
and that, for $r,q\in I_i$,
\begin{equation}
\label{cell_g_increment_new}
\|g(r)-g(q)\|_V
\leq
C\left(
|r-q|
+
\Delta t\,|r-q|^\nu
\right),
\end{equation}
where $\nu>\frac12-H$. Then
\begin{equation}
\label{cell_covariance_main_new}
\|g\|_{\mathfrak H_H(V)}^2
\leq
C\Delta t^{2H+1},
\end{equation}
where $C$ is independent of $m$ and $\Delta t$.
\end{lem}

\begin{proof}
For each $i=0,\ldots,m-1$, let $g_i=g(t_i^+)$, and define\\ $g^{\mathrm{pc}}(r)=\sum_{i=0}^{m-1}g_i\mathbf 1_{I_i}(r)$. For the canonical Hilbert space of the fractional Brownian motion,
\begin{equation}
\label{indicator_cov_new}
\left\langle
\mathbf 1_{I_i},
\mathbf 1_{I_j}
\right\rangle_{\mathfrak H_H}
=
\Delta t^{2H}\rho_H(i-j),
\end{equation}
where $\rho_H(k)=\frac12\left(|k+1|^{2H}-2|k|^{2H}+|k-1|^{2H}\right)$. Since $H<1/2$, $\rho_H(k)=H(2H-1)|k|^{2H-2}+
O(|k|^{2H-3})$ as $|k|\rightarrow\infty$. Therefore, $\sum_{k\in\mathbb Z}|\rho_H(k)|<\infty$.
Hence the inequality $|\langle g_i,g_j\rangle_V|\leq\frac12\left(\|g_i\|_V^2+\|g_j\|_V^2\right)$,
together with \eqref{indicator_cov_new} and \eqref{cell_g_bound_new}, gives
\begin{align}
\left\|
g^{\mathrm{pc}}
\right\|_{\mathfrak H_H(V)}^2
&=
\Delta t^{2H}
\sum_{i,j=0}^{m-1}
\langle g_i,g_j\rangle_V
\rho_H(i-j)\nonumber\\
&\leq
C\Delta t^{2H}
\sum_{i=0}^{m-1}
\|g_i\|_V^2
\nonumber\\
&\leq
C m\Delta t^{2H+2}\leq
C\Delta t^{2H+1}.
\label{pc_part_final_new}
\end{align}
Let $\widetilde g = g - g^{\mathrm{pc}}$, then  $g(r)= g^{\mathrm{pc}}(r) + \widetilde g(r)$, $0\leq r\leq t_m$. Hence for $r\in I_i$, $\widetilde g(r) =g(r)-g_i$. Taking $q=t_i^+$ in \eqref{cell_g_increment_new}, we obtain
\begin{equation}
\label{g_tilde_point_new}
\|\widetilde g(r)\|_V
\leq
C
\left[
(r-t_i)
+
\Delta t(r-t_i)^\nu
\right].
\end{equation}

Moreover, for $r,q\in I_i$,
\begin{equation}
\label{g_tilde_increment_new}
\|\widetilde g(r)-\widetilde g(q)\|_V
\leq
C
\left(
|r-q|
+
\Delta t|r-q|^\nu
\right).
\end{equation}

By the increment characterization of the canonical Hilbert space
for $H<1/2$,
\begin{align*}
\|\widetilde g\|_{\mathfrak H_H(V)}^2
\leq {}&
C_H
\int_0^{t_m}\int_0^{t_m}
\|\widetilde g(r)-\widetilde g(q)\|_V^2
|r-q|^{2H-2}\,dq\,dr
\nonumber\\
&+
C_H
\int_0^{t_m}
\|\widetilde g(r)\|_V^2
\left[
r^{2H-1}
+
(t_m-r)^{2H-1}
\right]dr\nonumber\\
\leq {}& C_H\sum_{i=0}^{m-1}
\int_{I_i}\int_{I_i}
\|\widetilde g(r)-\widetilde g(q)\|_V^2
|r-q|^{2H-2}\,dq\,dr \nonumber\\
&+C_H\sum_{\substack{i,j=0\\i\neq j}}^{m-1}\int_{I_i}\int_{I_j}\|\widetilde g(r)-\widetilde g(q)\|_V^2
|r-q|^{2H-2}\,dq\,dr
\nonumber\\
&+
C_H
\int_0^{t_m}
\|\widetilde g(r)\|_V^2
\left[
r^{2H-1}
+
(t_m-r)^{2H-1}
\right]dr\nonumber\\
&=: \mathcal J_{1, \mathrm{diag}} + \mathcal J_{1, \mathrm{off}} + \mathcal J_{2}.
\end{align*}
We first consider pairs belonging to the same cell. Since $\nu>\frac12-H$, we have $2H+2\nu-2>-1$. Thus from \eqref{g_tilde_increment_new},
\begin{align}
\mathcal J_{1, \mathrm{diag}}& := 
C_H\sum_{i=0}^{m-1}
\int_{I_i}\int_{I_i}
\|\widetilde g(r)-\widetilde g(q)\|_V^2
|r-q|^{2H-2}\,dq\,dr
\nonumber\\
&\leq
C
\sum_{i=0}^{m-1}
\int_{I_i}\int_{I_i}
|r-q|^{2H}\,dq\,dr
\nonumber\\
&\quad+
C\Delta t^2
\sum_{i=0}^{m-1}
\int_{I_i}\int_{I_i}
|r-q|^{2H+2\nu-2}\,dq\,dr\nonumber\\
&\leq Cm\Delta t^{2H+2}+ Cm\Delta t^{2H+2\nu+2}\nonumber\\
&\leq C m\Delta t^{2H+2}\leq C \Delta t^{2H+1}.
\label{same_cell_final_new}
\end{align}
We now estimate the contribution corresponding to
$r\in I_i$ and $q\in I_j$ with $i\neq j$. From
\eqref{g_tilde_point_new}, since $0\le r-t_i\le\Delta t$, we have $\|\widetilde g(r)\|_V
\le C\left(\Delta t+\Delta t^{1+\nu}\right)\le C\Delta t$.
Hence, for $r\in I_i$ and $q\in I_j$, $i\neq j$, $\|\widetilde g(r)-\widetilde g(q)\|_V
\le \|\widetilde g(r)\|_V+\|\widetilde g(q)\|_V\le C\Delta t$. Therefore, following a similar reasoning as the estimate \eqref{L21_off_final}, we get
\begin{align}
\mathcal J_{1,\mathrm{off}}
&:=
C_H
\sum_{\substack{i,j=0\\i\neq j}}^{m-1}
\int_{I_i}\int_{I_j}
\|\widetilde g(r)-\widetilde g(q)\|_V^2
|r-q|^{2H-2}\,dq\,dr
\nonumber\\
&\le
C\Delta t^2
\sum_{\substack{i,j=0\\i\neq j}}^{m-1}
\int_{I_i}\int_{I_j}
|r-q|^{2H-2}\,dq\,dr\nonumber\\
&\le
Cm\Delta t^{2H+2}\le
C\Delta t^{2H+1}.
\label{rem_off_final}
\end{align}
Finally, since $H<1/2$, we have $2H+1<2$. Thus the last contribution satisfies
\begin{align}
\mathcal J_{2}&:= C_H\int_0^{t_m}\|\widetilde g(r)\|_V^2\left[r^{2H-1}+(t_m-r)^{2H-1}\right]dr
\nonumber\\
&\leq
C\Delta t^2\leq
C\Delta t^{2H+1}.
\label{bound_cont}
\end{align}
Combining \eqref{pc_part_final_new}, \eqref{same_cell_final_new},
\eqref{rem_off_final}, and \eqref{bound_cont} yields
\eqref{cell_covariance_main_new}.
\end{proof}

With the preceding auxiliary estimates at hand, we are now in a
position to estimate the accumulated nonlinear drift defect.

\begin{lem}[Accumulated nonlinear drift defect]
\label{improved_drift_defect}

Under Assumptions \ref{initdata}--\ref{non}, for every \\
$m=1,\ldots,M$, the accumulated nonlinear drift defect satisfies
\begin{equation}
\label{drift_defect_final}
\left\|
\sum_{i=0}^{m-1}
\int_{t_i}^{t_{i+1}}
S_h(t_m-s)P_h
\left[
F(X^h(s))-F(X^h(t_i))
\right]ds
\right\|_{L^2(\Omega;\mathcal H)}
\leq
C\Delta t^{\frac{2H+\beta}{2}},
\end{equation}
where $C$ is independent of $h$, $\Delta t$, and $m$.
\end{lem}


\begin{proof}

For $s\in[t_i,t_{i+1}]$, define
\begin{equation}
\label{Ihis_def}
I_{i,s}^h=\int_0^1F'\left(X^h(t_i)+\lambda\left[X^h(s)-X^h(t_i)\right]\right)d\lambda.
\end{equation}

The fundamental theorem of calculus in Banach spaces gives the
exact identity
\begin{equation}
\label{Taylor_integral_exact}
F(X^h(s))-F(X^h(t_i))
=
I_{i,s}^h
\left[
X^h(s)-X^h(t_i)
\right].
\end{equation}
Since $F'\in C_b$, we have $\|I_{i,s}^h\|_{L(\mathcal H)}\leq C$ uniformly in $i$, $s$, and $h$. For $s\in[t_i,t_{i+1}]$, the mild representation gives
\begin{align}
X^h(s)-X^h(t_i)&=\left[S_h(s-t_i)-I\right]X^h(t_i)\nonumber\\
&\quad+\int_{t_i}^{s}S_h(s-r)P_hF(X^h(r))\,dr
+\int_{t_i}^{s}S_h(s-r)P_h\phi\,\,d B_Q^H(r).
\label{increment_DI_NI}
\end{align}
Substituting \eqref{increment_DI_NI} into
\eqref{Taylor_integral_exact} yields
\begin{align}
\label{EF_decomp}
\mathcal E_m^F&:= \sum_{i=0}^{m-1}
\int_{t_i}^{t_{i+1}}
S_h(t_m-s)P_h
\left[
F(X^h(s))-F(X^h(t_i))
\right]ds\nonumber\\
&= \sum_{i=0}^{m-1}
\int_{t_i}^{t_{i+1}}
S_h(t_m-s)I_{i,s}^h\left[S_h(s-t_i)-I\right]X^h(t_i)\,ds\nonumber\\
&\quad+\sum_{i=0}^{m-1}
\int_{t_i}^{t_{i+1}}
S_h(t_m-s)I_{i,s}^h\int_{t_i}^{s}S_h(s-r)P_hF(X^h(r))\,dr\,ds\nonumber\\
&\quad+\sum_{i=0}^{m-1}
\int_{t_i}^{t_{i+1}}
S_h(t_m-s)I_{i,s}^h\int_{t_i}^{s}S_h(s-r)P_h\phi\,dB_Q^H(r)\,ds\nonumber\\
& =: \mathcal E_{m,1}^F+\mathcal E_{m,2}^F+\mathcal E_{m,3}^F.
\end{align}
Taking the norm on both sides of \eqref{EF_decomp} and using the triangle inequality yields
\begin{equation*}
\left\|\mathcal E_m^F\right\|_{L^2(\Omega;\mathcal H)}
\leq
\left\|\mathcal E_{m,1}^F\right\|_{L^2(\Omega;\mathcal H)}
+
\left\|\mathcal E_{m,2}^F\right\|_{L^2(\Omega;\mathcal H)}
+
\left\|\mathcal E_{m,3}^F\right\|_{L^2(\Omega;\mathcal H)}.
\end{equation*}
Using the triangle inequality, the smoothing properties of the semigroup,
the semidiscrete version of Theorem \ref{regfrac} \eqref{sparegfrac} and Proposition \ref{semigroup} with $\gamma=\frac{2H+\beta}{2}$ yields
\begin{align}
\left\|\mathcal E_{m,1}^F\right\|_{L^2(\Omega;\mathcal H)}
&\leq
\sum_{i=0}^{m-1}
\int_{t_i}^{t_{i+1}}
\left\|
S_h(t_m-s)I_{i,s}^h\left[S_h(s-t_i)-I\right]X^h(t_i)
\right\|_{L^2(\Omega;\mathcal H)}
\,ds
\nonumber\\
&\leq
\sum_{i=0}^{m-1}
\int_{t_i}^{t_{i+1}}
\left\|
S_h(t_m-s)I_{i,s}^h
\right\|_{L(\mathcal H)}
\nonumber\\
&\quad\times
\left\|
\left[S_h(s-t_i)-I\right]
A_h^{-\frac{2H+\beta}{2}}
\right\|_{L(\mathcal H)}
\left\|
A_h^{\frac{2H+\beta}{2}}X^h(t_i)
\right\|_{L^2(\Omega;\mathcal H)}
\,ds
\nonumber\\
&\leq
C
\sum_{i=0}^{m-1}
\int_{t_i}^{t_{i+1}}
(s-t_i)^{\frac{2H+\beta}{2}}
\,ds\leq
C\Delta t^{\frac{2H+\beta}{2}}.
\label{norm_EF_decomp1}
\end{align}
Using the triangle inequality, the smoothing properties of the semigroup, Assumption \ref{non} \eqref{ass_non_lin3} and the semidiscrete version of \eqref{mean_square_bound}, it holds that
\begin{align}
\left\|\mathcal E_{m,2}^F\right\|_{L^2(\Omega;\mathcal H)}
&\leq
\sum_{i=0}^{m-1}
\int_{t_i}^{t_{i+1}}
\int_{t_i}^{s}
\left\|
S_h(t_m-s)I_{i,s}^h
S_h(s-r)P_hF(X^h(r))
\right\|_{L^2(\Omega;\mathcal H)}
\,dr\,ds
\nonumber\\
&\leq
C
\sum_{i=0}^{m-1}
\int_{t_i}^{t_{i+1}}
(s-t_i)\,ds
\nonumber\\
&\leq
C\Delta t\leq
C\Delta t^{\frac{2H+\beta}{2}}.
\label{norm_EF_decomp2}
\end{align}
By the stochastic Fubini theorem,
\[
\begin{aligned}
 \mathcal E_{m,3}^F &=
\sum_{i=0}^{m-1}
\int_{t_i}^{t_{i+1}}
\left(
\int_{t_i}^{s}
S_h(t_m-s)I_{i,s}^{h}
S_h(s-r)P_h\phi\,dB_Q^H(r)
\right)ds
\nonumber\\
&=
\sum_{i=0}^{m-1}
\int_{t_i}^{t_{i+1}}
\left(
\int_{r}^{t_{i+1}}
S_h(t_m-s)I_{i,s}^{h}
S_h(s-r)P_h\phi\,ds
\right)
\delta B_Q^H(r)\nonumber\\
&=\sum_{i=0}^{m-1}
\int_{t_i}^{t_{i+1}}K_m(r)\,\delta B_Q^H(r),   
\end{aligned}
\]
where, for $r\in I_i$,
\begin{equation*}
K_m(r)
=
\int_r^{t_{i+1}}
S_h(t_m-s)P_h
I_{i,s}^h
S_h(s-r)P_h\phi\,ds.
\end{equation*}
Using the smoothing properties of the semigroup and Assumption \ref{noise} with $\kappa=0$, we obtain
\[
\begin{aligned}
\|K_m(r)\|_{L^2(\Omega;L_2^0)}
&\leq
C
\int_r^{t_{i+1}}
\|S_h(s-r)P_h\phi\|_{L_2^0}ds
\nonumber\\
&\leq
C(t_{i+1}-r)\leq
C\Delta t.
\end{aligned}
\]
Let $q<r$ belong to the same interval $I_i$. the triangle inequality, the semigroup property, Proposition \ref{semigroup} with
$\gamma=\frac{1+\beta}{2}$, the boundedness of the
semigroup, the stability of the finite element projection, and
Assumption \ref{noise} yield
\[
\begin{aligned}
&\|K_m(r)-K_m(q)\|_{L^2(\Omega;L_2^0)}\\
\leq {}&
\left\|\int_q^r
S_h(t_m-s)P_hI_{i,s}^h
S_h(s-q)P_h\phi\,ds\right\|_{L^2(\Omega;L_2^0)}
\nonumber\\
&+
\left\|\int_r^{t_{i+1}}
S_h(t_m-s)P_hI_{i,s}^h
S_h(s-r)\left(I-S_h(r-q)\right)P_h\phi\,ds\right\|_{L^2(\Omega;L_2^0)}
\nonumber\\
&\leq
C
\left(
|r-q|
+
\Delta t|r-q|^{\frac{1+\beta}{2}}
\right).
\label{Km_I_increment}
\end{aligned}
\]
Since $\frac{1+\beta}{2}>\frac12-H$, Lemma \ref{lem_cellwise_covariance} applied with
$V=L^2(\Omega;L_2^0)$ gives
\begin{equation}
\label{Km_I_HH}
\mathbb E
\|K_m\|_{\mathfrak H_H(L_2^0)}^2
\leq
C\Delta t^{2H+1}.
\end{equation}


We next estimate $D K_m$. From \eqref{Ihis_def} and the
Malliavin chain rule, we have
\begin{equation*}
D_\theta I_{i,s}^h
=
\int_0^1
F''(Y_{i,s,\lambda}^h)
\left[
D_\theta Y_{i,s,\lambda}^h,\cdot
\right]d\lambda,
\end{equation*}
where $Y_{i,s,\lambda}^h=(1-\lambda)X^h(t_i)+\lambda X^h(s)$, and $D_\theta Y_{i,s,\lambda}^h=(1-\lambda)D_\theta X^h(t_i)+\lambda D_\theta X^h(s)$.
Therefore, using the boundedness of $F''$ and Lemma
\ref{lem_malliavin_Xh}, we obtain
\begin{align*}
\left\|
D I_{i,s}^h
\right\|_{L^4(\Omega;
\mathfrak H_H(L(\mathcal H)))}
&\leq
C\int_0^1
\left\|
D Y_{i,s,\lambda}^h
\right\|_{L^4(\Omega;
\mathfrak H_H(L_2^0))}
\,d\lambda
\nonumber\\
&\leq
C\left(
\left\|
D X^h(t_i)
\right\|_{L^4(\Omega;
\mathfrak H_H(L_2^0))}
+
\left\|
D X^h(s)
\right\|_{L^4(\Omega;
\mathfrak H_H(L_2^0))}
\right)
\nonumber\\
&\leq C.
\end{align*}
Hence,
\begin{equation}
\label{DIhis_sup}
\sup_{\substack{0\leq i\leq m-1\\
t_i\leq s\leq t_{i+1}}}
\left\|
D I_{i,s}^h
\right\|_{L^4(\Omega;
\mathfrak H_H(L(\mathcal H)))}
\leq C.
\end{equation}

Consequently, for $r\in I_i$, differentiating $K_m(r)$ with
respect to the underlying fractional Brownian motion yields
\begin{equation*}
D_\theta K_m(r)
=
\int_r^{t_{i+1}}
S_h(t_m-s)P_h
D_\theta I_{i,s}^h
S_h(s-r)P_h\phi\,ds.
\end{equation*}
Using Minkowski's inequality, the boundedness of the semigroup,
Assumption \ref{noise}, and \eqref{DIhis_sup}, we obtain
\begin{align}
\left\|
D K_m(r)
\right\|_{L^2(\Omega;
\mathfrak H_H(L_2^0))}
&\leq
C\int_r^{t_{i+1}}
\left\|
D I_{i,s}^h
\right\|_{L^2(\Omega;
\mathfrak H_H(L(\mathcal H)))}
\,ds
\nonumber\\
&\leq
C(t_{i+1}-r)\leq
C\Delta t.
\label{DKm_pointwise}
\end{align}

It remains to estimate the fractional covariance norm of $D K_m$.
For $r,q\in I_i$, repeating the same decomposition used above for
$K_m$, together with \eqref{DIhis_sup}, yields
\begin{equation}
\label{DKm_increment}
\left\|
D K_m(r)-D K_m(q)
\right\|_{L^2(\Omega;
\mathfrak H_H(L_2^0))}
\leq
C\left(
|r-q|
+
\Delta t\,|r-q|^{\frac{1+\beta}{2}}
\right),
\end{equation}
Therefore, Lemma \ref{lem_cellwise_covariance}, applied with $V=L^2\left(\Omega;\mathfrak H_H(L_2^0)\right)$, together with \eqref{DKm_pointwise} and
\eqref{DKm_increment}, gives
\begin{equation}
\label{DKm_canonical_final}
\left\|
D K_m
\right\|_{L^2(\Omega;
\mathfrak H_H\otimes\mathfrak H_H(L_2^0))}^2\leq
C\Delta t^{2H+1}.
\end{equation}
Combining \eqref{Km_I_HH} and \eqref{DKm_canonical_final} with \eqref{divergence_estimate}, we obtain
\begin{align*}
\|\mathcal E_{m,3}^F\|_{L^2(\Omega;\mathcal H)}^2
&=
\|\delta^H(K_m)\|_{L^2(\Omega;\mathcal H)}^2
\nonumber\\
&\leq
C
\left[
\mathbb E
\|K_m\|_{\mathfrak H_H(L_2^0)}^2
+
\mathbb E
\|DK_m\|_{
\mathfrak H_H\otimes
\mathfrak H_H(L_2^0)
}^2
\right]
\nonumber\\
&\leq
C\Delta t^{2H+1}\leq C\Delta t^{2H+\beta}.
\end{align*}

Hence 
\begin{align}
\|\mathcal E_{m,3}^F\|_{L^2(\Omega;\mathcal H)}
\leq C\Delta t^{\frac{2H+\beta}2}.
\label{norm_EF_decomp3}
\end{align}

Finally, combining \eqref{norm_EF_decomp1}, \eqref{norm_EF_decomp2} and
\eqref{norm_EF_decomp3} proves \eqref{drift_defect_final}
\end{proof}

\subsection{Strong convergence of the fully discrete scheme}

We are now in a position to establish the main convergence result.

\begin{thm}[Main strong convergence result]
\label{main}

Let $H\in\left(0,\frac12\right)$, and $1-2H<\beta\leq1$. Assume that Assumptions \ref{initdata}--\ref{non} hold. Let $X(t_m)$ be the exact mild solution of \eqref{pb} at $t=t_m$ and let $X_m^h$ be the fully discrete approximation defined by
\eqref{SETD1}. Then 
\begin{equation*}
\left\|X(t_m)-X_m^h\right\|_{L^2(\Omega;\mathcal H)}
\leq C\left(h^{2H+\beta}+\Delta t^{\frac{2H+\beta}{2}}\right).
\end{equation*}
The constant $C$ is independent of $h$, $\Delta t$, and $m$.
\end{thm}

\begin{proof}

By the triangle inequality,
\begin{align*}
\left\|
X(t_m)-X_m^h
\right\|_{L^2(\Omega;\mathcal H)}
\leq
\left\|X(t_m)-X^h(t_m)\right\|_{L^2(\Omega;\mathcal H)}+
\left\|X^h(t_m)-X_m^h\right\|_{L^2(\Omega;\mathcal H)}.
\end{align*}

By Theorem \ref{spaconv},
\begin{equation}
\label{space_error_used}
\left\|
X(t_m)-X^h(t_m)
\right\|_{L^2(\Omega;\mathcal H)}
\leq
C h^{2H+\beta}.
\end{equation}

It remains to estimate the temporal error. Subtracting \eqref{SETD1} from \eqref{soldis} and then taking the $L^2(\Omega;\mathcal H)$-norm, the triangle inequality yields
\begin{align}
\label{temporal_split}
\left\|
X^h(t_m)-X_m^h
\right\|_{L^2(\Omega;\mathcal H)}
&\leq
\left\|
\int_0^{t_m}
S_h(t_m-s)P_h
\left[
F(X^h(s))-F(X_{[s]^m}^h)
\right]ds
\right\|_{L^2(\Omega;\mathcal H)}
\nonumber\\
&\quad+
\left\|
\int_0^{t_m}
\left[
S_h(t_m-s)-S_h(t_m-[s])
\right]
P_h\phi\,dB_Q^H(s)
\right\|_{L^2(\Omega;\mathcal H)}\nonumber\\
&=:
L_1+L_2,
\end{align}
For $L_1$, we add and subtract $F(X^h([s]))$:
\begin{align*}
L_1
&\leq
\left\|
\int_0^{t_m}
S_h(t_m-s)P_h
\left[
F(X^h(s))-F(X^h([s]))
\right]ds
\right\|_{L^2}
\nonumber\\
&\quad+
\left\|
\int_0^{t_m}
S_h(t_m-s)P_h
\left[
F(X^h([s]))-F(X_{[s]^m}^h)
\right]ds
\right\|_{L^2}
\nonumber\\
&=:L_{11}+L_{12}.
\end{align*}
By Lemma \ref{improved_drift_defect}, $\displaystyle L_{11}
\leq C\Delta t^{\frac{2H+\beta}{2}}$. For $L_{12}$, the global Lipschitz continuity of $F$ \eqref{ass_non_lin_lipschitz} and the boundedness of $S_h(t)$ yield $\displaystyle L_{12}\leq C\sum_{i=0}^{m-1}\int_{t_i}^{t_{i+1}}\left\|X^h(t_i)-X_i^h\right\|_{L^2(\Omega;\mathcal H)}ds\leq C\Delta t
\sum_{i=0}^{m-1}\left\|X^h(t_i)-X_i^h\right\|_{L^2(\Omega;\mathcal H)}$. Thus,
\begin{equation}
\label{L1final}
L_1
\leq
C\Delta t^{\frac{2H+\beta}{2}}
+
C\Delta t
\sum_{i=0}^{m-1}
\left\|
X^h(t_i)-X_i^h
\right\|_{L^2(\Omega;\mathcal H)}.
\end{equation}
On the other hand, Lemma \ref{stoch_quad_error} gives
\begin{equation}
\label{L2final}
L_2
\leq
C\Delta t^{\frac{2H+\beta}{2}}.
\end{equation}
Combining \eqref{temporal_split}, \eqref{L1final}, and
\eqref{L2final}, we obtain
\begin{align*}
\left\|
X^h(t_m)-X_m^h
\right\|_{L^2(\Omega;\mathcal H)}
\leq
C\Delta t^{\frac{2H+\beta}{2}}+
C\Delta t
\sum_{i=0}^{m-1}
\left\|
X^h(t_i)-X_i^h
\right\|_{L^2(\Omega;\mathcal H)}.
\end{align*}
The discrete Gronwall inequality therefore yields
\begin{equation}
\label{time_err2}
\left\|
X^h(t_m)-X_m^h
\right\|_{L^2(\Omega;\mathcal H)}
\leq
C
\Delta t^{\frac{2H+\beta}{2}}.
\end{equation}
Finally, combining \eqref{space_error_used} and \eqref{time_err2}
gives
\[
\left\|
X(t_m)-X_m^h
\right\|_{L^2(\Omega;\mathcal H)}
\leq
C
\left(
h^{2H+\beta}
+
\Delta t^{\frac{2H+\beta}{2}}
\right).
\]
This completes the proof of Theorem \ref{main}.
\end{proof}

\section{Numerical simulations}
\label{numerik}

In this section, we present numerical simulations to illustrate the
strong convergence result obtained in Theorem~\ref{main} and quantify the mean of the solution for Hurst parameters $H \in (0, \frac{1}{2})$ using Monte Carlo method. We consider
the stochastic advection--diffusion--reaction equation
\begin{equation*}
dX(t)
=
\left[
\nabla\cdot(D\nabla X(t))
-\nabla\cdot(\mathbf{q}X(t))
-\frac{X(t)}{1+X(t)^2}
\right]dt
+
\phi\,dB_Q^H(t)
\end{equation*}
in $\Lambda=[0,1]\times[0,1]$, where $D=
\begin{pmatrix}
10^{-2} & 0\\
0 & 10^{-3}
\end{pmatrix}$, $\phi=2$ and $X_0=0$.
We prescribe the Dirichlet boundary condition $X=1$ at $x=0$ and
homogeneous Neumann boundary conditions on the remaining parts of
$\partial\Lambda$. The nonlinear function is given by $\displaystyle F(v)=-\frac{v}{1+v^2}$.
It is straightforward to verify that $F$ is twice differentiable and
satisfies Assumption \ref{non}.
The velocity field $\mathbf{q}$ is obtained by solving Darcy's problem
\[
\left\{
\begin{aligned}
\nabla\cdot\mathbf{q} &= 0,\\
\mathbf{q} &= -k\nabla p.
\end{aligned}
\right.
\]
with
$p=1$ at $x=0$, $p=0$ at $x=1$
and homogeneous no-flow conditions at $y=0$ and $y=1$. The
permeability $k$ is modeled as a lognormal random field using a
truncated Karhunen--Lo\`eve expansion with correlation parameters
$b_1=0.4$ and $b_2=0.2$. We retain $30$ KL modes and use the same
realization of the permeability field throughout the simulations.
The corresponding log-permeability field and Darcy velocity
streamlines are shown in Fig.~\ref{fig:num_fields}.
\begin{figure}[htbp]
\centering
\begin{minipage}{0.48\textwidth}
    \centering
    \includegraphics[width=\textwidth]
    {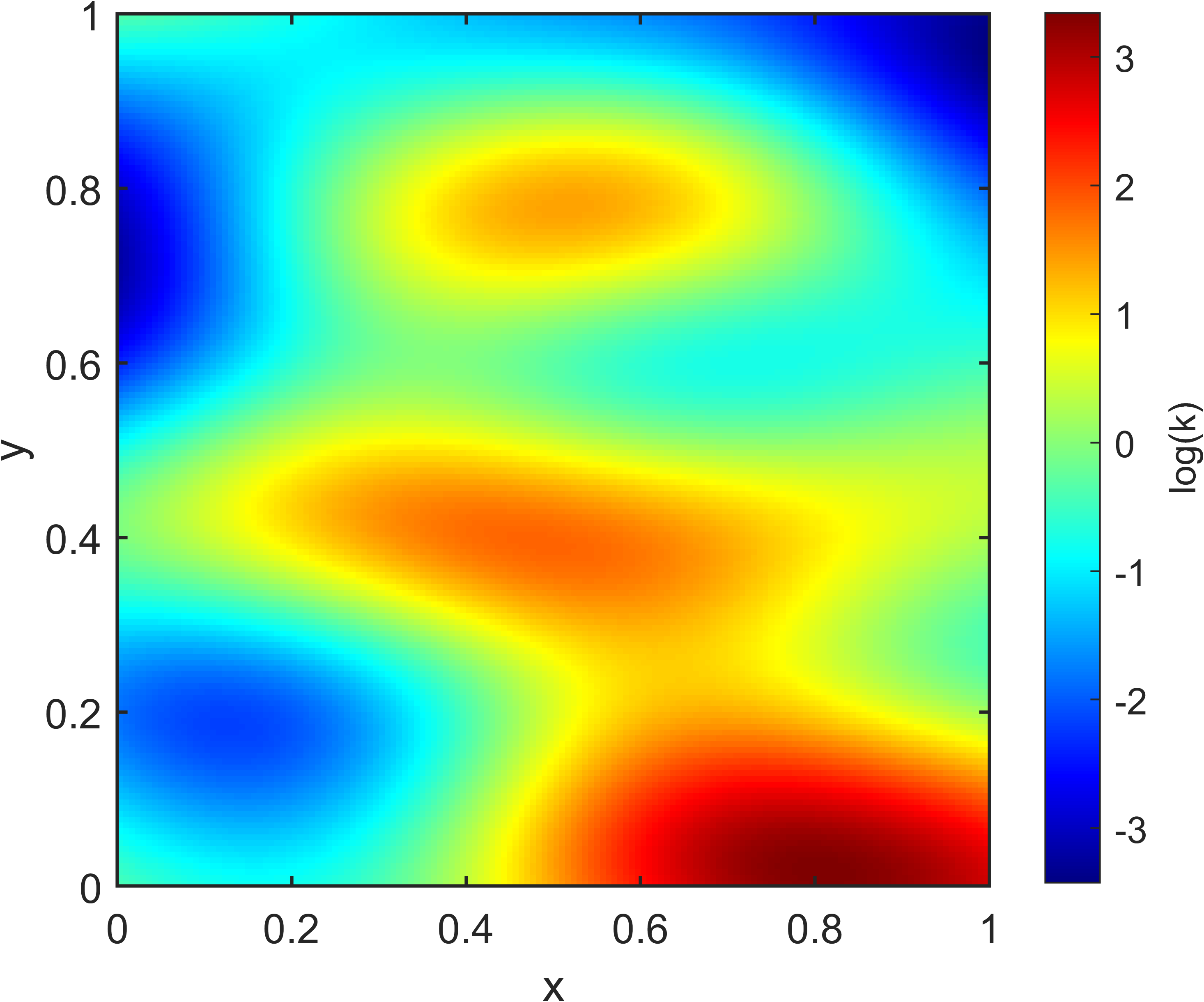}\\
    \textnormal{(a)}
\end{minipage}
\hfill
\begin{minipage}{0.48\textwidth}
    \centering
    \includegraphics[width=0.8\textwidth]
    {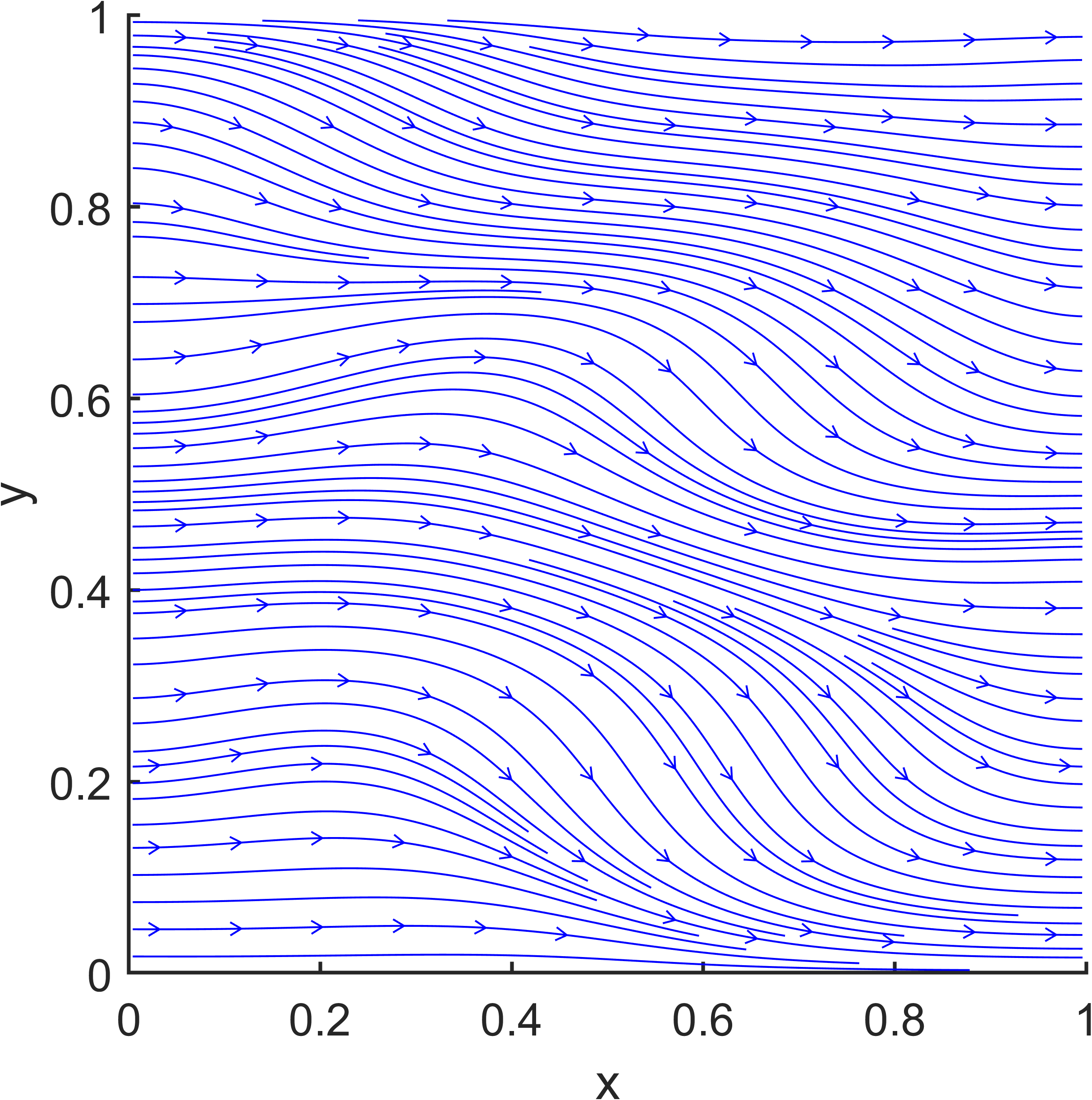}\\
    \textnormal{(b)}
\end{minipage}
\caption{(a) Stochastic log-permeability field and
(b) streamlines of the corresponding Darcy velocity field.}
\label{fig:num_fields}
\end{figure}

For the fractional noise, we use the truncated expansion
\[
B_Q^H(t)
=
\sum_{\substack{0\leq i,j\leq60\\ i+j>0}}
\sqrt{q_{ij}}\,e_{ij}\beta_{ij}^H(t),
\qquad
q_{ij}=(i+j)^{-6.1},
\]
where $\{\beta_{ij}^H\}$ are independent real-valued fractional
Brownian motions and $\{e_{ij}\}$ denotes the tensor-product cosine
basis. The exponent $6.1$ ensures a sufficiently fast spectral decay to
satisfy the spatial regularity condition
$A^\kappa\phi\in L_2^0$ for every $\kappa\in[0,1]$.
Indeed, for $q_{ij}=(i+j)^{-r}$, this condition holds provided
$r>4\kappa+2$. Since $4\kappa+2\leq6$ for
$\kappa\in[0,1]$, the choice $r=6.1$ satisfies this requirement. The fractional Brownian increments are generated using the
Davies--Harte circulant embedding method \cite{dieker2004simulation}. We consider $H\in\{0.20,0.30,0.40,0.49\}$. We discretize the spatial domain using a uniform $150\times150$ finite-volume grid, which is well suited for the conservative treatment of the Darcy fluxes. We take $T=1$ and investigate the temporal convergence of SETDM1. Since the exact solution is unavailable, the numerical solution with
$\Delta t_{\mathrm{ref}}=1/1600$ is taken as the reference solution.
We test $\Delta t=1/25,\,1/50,\,1/100,\,1/200,\,1/400$, using
$M=200$ fractional Brownian sample paths. The errors are evaluated at
$T=1$ in the root mean-square $L^2(\Lambda)$ norm. For each
realization, the reference and coarse solutions are driven by the
same fractional Brownian path.
Figure~\ref{fig:convergence} shows the strong errors as functions of
the time step. The observed convergence orders are $0.72$, $0.83$,
$0.94$, and $1.01$ for $H=0.20$, $0.30$, $0.40$, and $0.49$,
respectively. For $\beta=1$, Theorem~\ref{main} predicts the temporal
convergence rate $H+\frac12$, which gives the theoretical orders $0.70$, $0.80$, $0.90$, and $0.99$, respectively. The observed rates are therefore in good agreement with the theoretical convergence rates.
\begin{figure}[htbp]
\centering
\includegraphics[width=0.5\textwidth]
{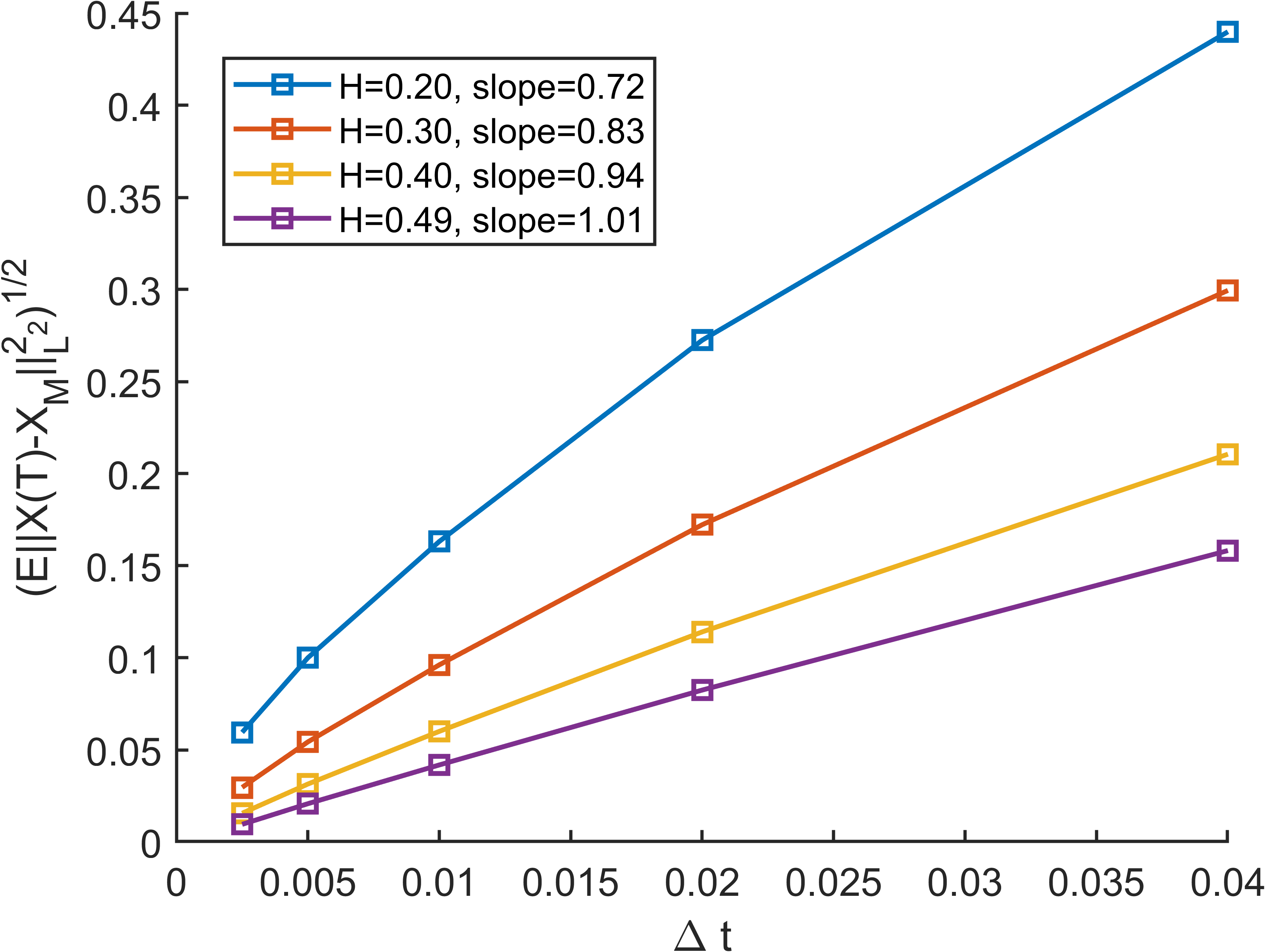}
\caption{Strong convergence of SETDM1 for
$H=0.20$, $0.30$, $0.40$, and $0.49$.}
\label{fig:convergence}
\end{figure}
Figure~\ref{fig:solutions} shows the deterministic solution and the
Monte Carlo mean of the stochastic solution at $T=1$. The
deterministic solution is obtained by suppressing the fractional
noise while keeping all the other parameters unchanged. Compared with the deterministic solution, the stochastic mean
preserves the main transport pathways induced by the heterogeneous
Darcy velocity field, while the fractional forcing changes the
concentration levels along these pathways. The comparison between
$H=0.20$ and $H=0.40$ also shows moderate changes in the intensity
and spatial distribution of the mean concentration.

\begin{figure}[htbp]
\centering
\begin{minipage}{0.32\textwidth}
    \centering
    \includegraphics[width=\textwidth]
    {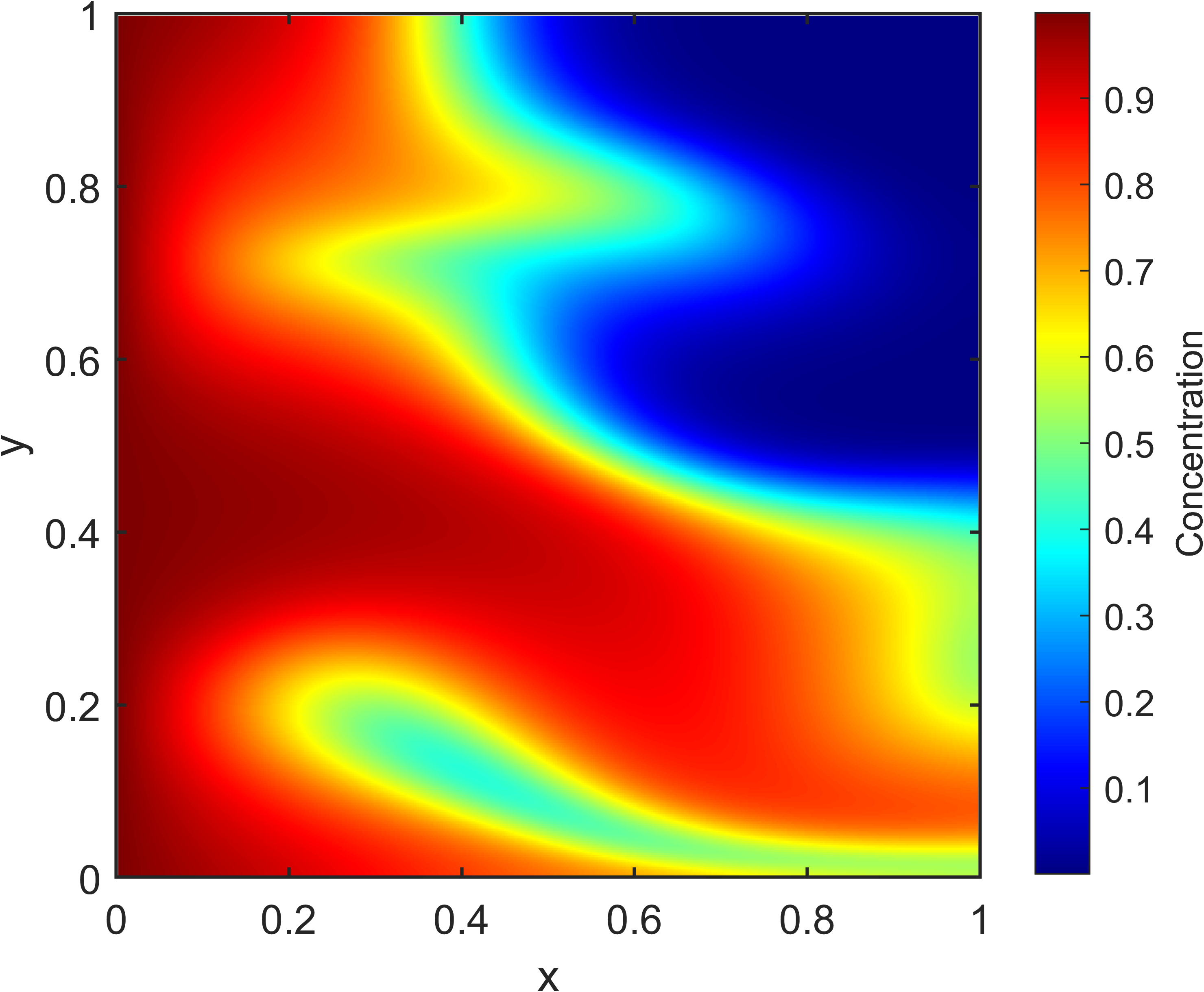}\\
    \textnormal{(a)}
\end{minipage}
\hfill
\begin{minipage}{0.32\textwidth}
    \centering
    \includegraphics[width=\textwidth]
    {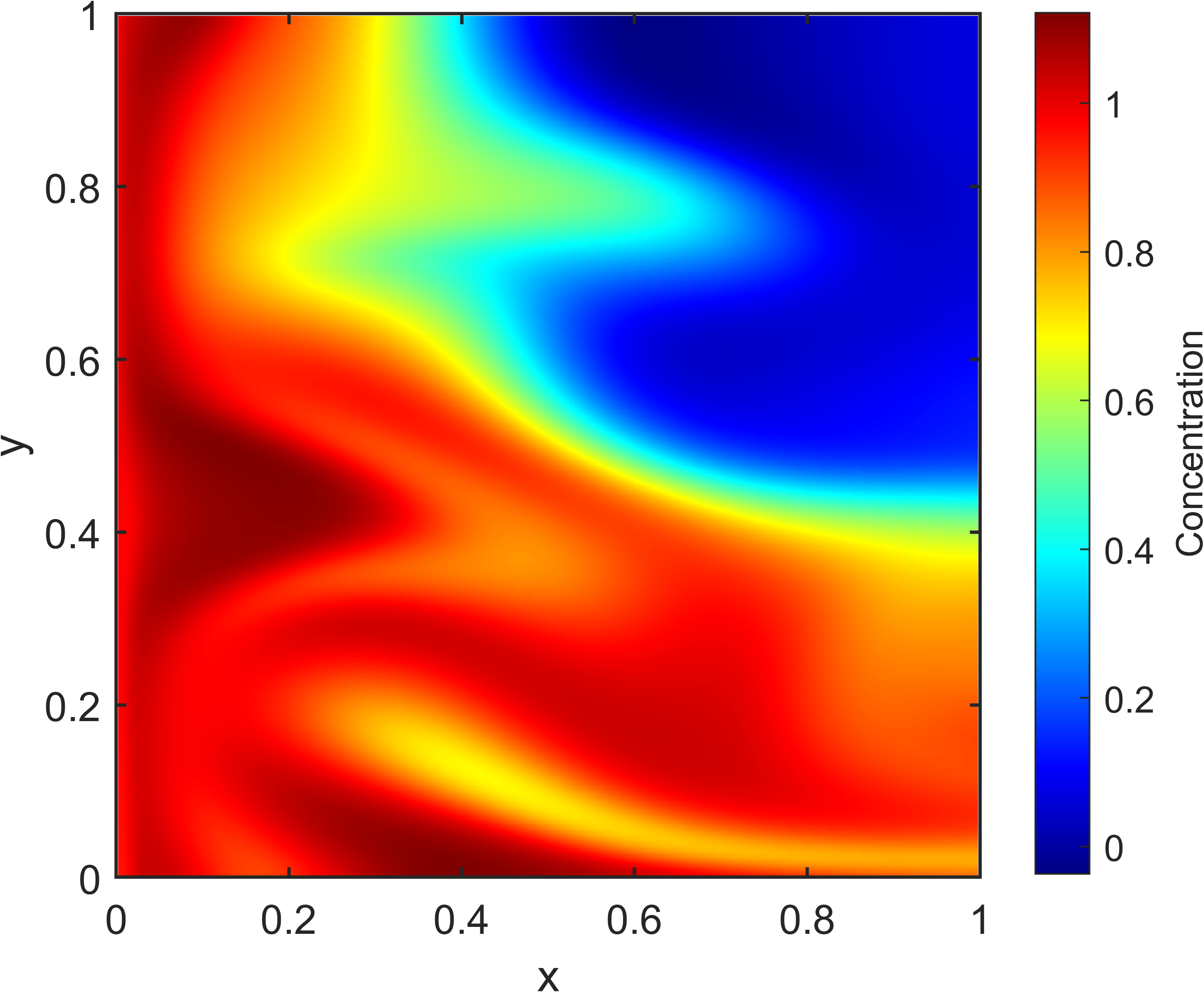}\\
    \textnormal{(b)}
\end{minipage}
\hfill
\begin{minipage}{0.32\textwidth}
    \centering
    \includegraphics[width=\textwidth]
    {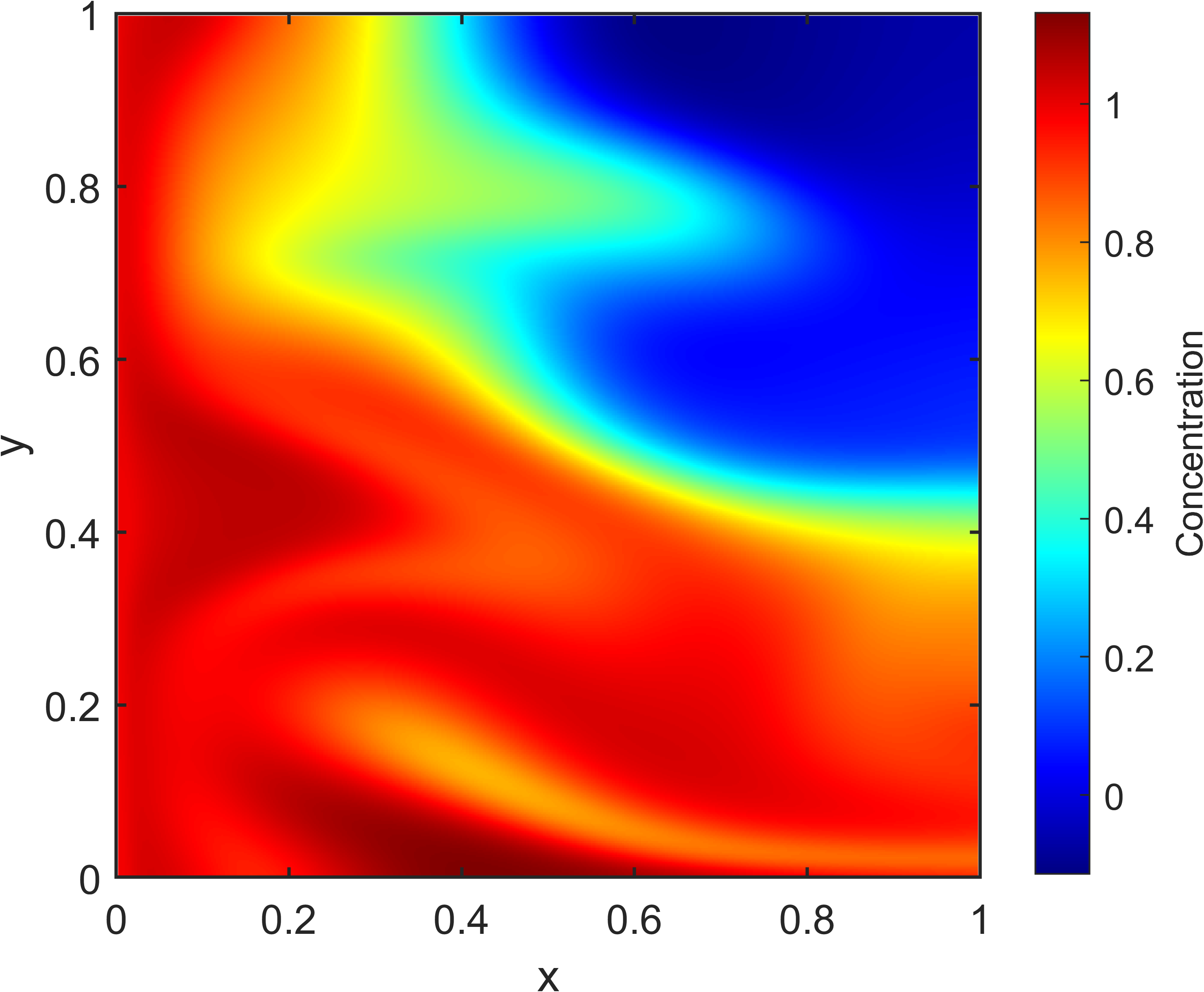}\\
    \textnormal{(c)}
\end{minipage}
\caption{Numerical solutions at $T=1$: (a) deterministic solution,
(b) mean stochastic solution for $H=0.20$, and
(c) mean stochastic solution for $H=0.40$.}
\label{fig:solutions}
\end{figure}

\bibliographystyle{plain}
\bibliography{references}
\end{document}